\documentclass[psamsfonts,11pt]{amsart}
\usepackage{mathpazo}
\usepackage{times}
\usepackage{graphicx}
\usepackage{psfrag}
\usepackage{mathrsfs}
\usepackage{color}
\usepackage{amsmath,amssymb, amsfonts,amsthm}
\usepackage{hyperref}
\usepackage[active]{srcltx}
\usepackage[normalem]{ulem}

\numberwithin{equation}{section}

\usepackage{thmtools}
\usepackage{thm-restate}

\newtheorem{maintheorem}{Theorem}

\newtheorem{maincorollary}[maintheorem]{Corollary}

\newtheorem{theorem}{Theorem}[section]

\newtheorem{corollary}[theorem]{Corollary}
\newtheorem{proposition}[theorem]{Proposition}
\newtheorem{lemma}[theorem]{Lemma}
\newtheorem{definition}[theorem]{Definition}

\newtheorem{remark}[theorem]{Remark}
\newtheorem{problem}{Problem}

\begin{document}


\title{Ergodic properties of partially hyperbolic diffeomorphisms with topological neutral center}
\author{Gabriel Ponce}
\address{Department of Mathematics, Universidade Estadual de Campinas, 
Campinas-SP, Brazil}
\email{gaponce@unicamp.br}

\date{}

\maketitle


\begin{abstract}
In this work we obtain some metric and ergodic properties of $C^{1+}$ partially hyperbolic diffeomorphisms with one-dimensional topological neutral center, mainly regarding the behavior of its center foliation. Based on a trichotomy for the center conditional measures of any invariant ergodic measure, we show that if these conditionals have full support, then the center foliation is leafwise absolutely continuous, the diffeomorphism is Bernoulli in the $C^{1+}$ case, and an invariance principle occurs in the sense that $M$ may be covered by a finite number of open sets where the system of center conditionals is continuous and $su$-invariant. Using this invariance principle we show that if a local accessibility hypothesis occurs then the center foliation must be as regular as the partially hyperbolic dynamics.

\end{abstract}
\setcounter{tocdepth}{2}
\tableofcontents

\section{Introduction}

A diffeomorphism $f:M \to M$ defined on a compact Riemannian manifold $M$ is said to be partially hyperbolic if there is a nontrivial splitting 
\[TM=E^s\oplus E^c\oplus E^u\]
such that 
\[Df(x)E^{\tau}(x)=E^{\tau}(f(x)), \; \tau\in \{s,c,u\}\]
and a Riemannian metric for which there are continuous positive functions $\mu,\hat{\mu},\nu,\hat{\nu}, \gamma, \hat{\gamma}$ with
\[\nu(p),\hat{\nu}(p) <1, \quad \text{and} \quad \mu(p)<\nu(p)<\gamma(p)<\hat{\gamma}(p)^{-1}<\hat{\nu}(p)^{-1}<\hat{\mu}(p)^{-1},\]
such that for any vector $v\in T_pM$,
\[\mu(p)||v||<||Df(p)\cdot v|| < \nu(p)||v||, \; \text{if} \; v\in E^s(p)\]
\[\gamma(p)||v|| < ||Df(p)\cdot v|| < \hat{\gamma}(p)^{-1}||v||, \; \text{if} \; v\in E^c(p)\]
\[\hat{\nu}(p)^{-1}||v||< ||Df(p)\cdot v||<\hat{\mu}(p)^{-1}||v||, \; \text{if} \; v\in E^u(p).\]

We say that $f$ is \textit{volume preserving}, or that $f$ is \textit{conservative}, if $f$ preserves a probability measure which is equivalent to the volume measure given by the Riemannian structure of $M$. The stable and unstable directions of $f$, $E^s$ and $E^u$ respectively, integrate to $f$-invariant foliations $\mathcal F^s$ and $\mathcal F^u$, called \textit{stable} and \textit{unstable} foliations of $f$ respectively. The center direction, however, is not necessarily integrable.

For a partially hyperbolic diffeomorphism $f:M\to M$ and for $E\subset TM$ be a $Df$-invariant bundle, we say that
\begin{itemize}
\item $f$ is Lyapunov stable in the direction $E$ if for any $\varepsilon>0$ there is $\delta>0$ such that for any $C^1$ path $\gamma$ tangent to $E$
\[\operatorname{length}(\gamma)<\delta \Rightarrow \operatorname{length}(f^n\gamma)<\varepsilon, \quad \forall n\geq 0;\]
\item $f$ has Lyapunov stable center if $f$ is Lyapunov stable in the direction $E^c$;
\item $f$ has topological neutral center if $f$ and $f^{-1}$ both have Lyapunov stable center.
\end{itemize}

By  \cite[Corollary $7.6$]{HHUSurvey} it is known that if a partially hyperbolic diffeomorphism $f:M \to M$ has topological neutral center, then it is \textit{dynamically coherent}, that is, both $E^{cs}:=E^c\oplus E^s$ and $E^{cu}:=E^c\oplus E^u$ integrate to $f$-invariant foliations. In this case the $f$-invariant foliation $\mathcal F^c:= \mathcal F^{cs}\cap \mathcal F^{cu}$ is tangent everywhere to the center direction and is called center foliation. A stronger condition, which implies Lyapunov stable center and consequently implies the integrability of $E^c$, is the \textit{neutral center} condition where one requires a uniform bound on the derivatives of $f^n$ along the center. More precisely, we say that $f$ has neutral center direction if there exists $K>1$ such that
\[\frac{1}{K} \leq ||Df^n|E^c(x)|| \leq K\]
for every $x\in M$ and any $n\in \mathbb Z$.
Not every $C^1$ partially hyperbolic diffeomorphism with topological neutral center has neutral center as one may see in \cite[Proposition 2.3]{BonattiZhang}.
The nomenclatures ``neutral center'' and ``topological neutral center'' appeared for the first time in \cite{Zhang} motivated by examples of such diffeomorphisms which appeared in the construction of anomalous partially hyperbolic diffeomorphisms, providing counterexamples to the so called Pujals' conjecture, given in \cite{BPP, BGP} and \cite{BZ}. We remark, however, that some of the examples obtained in these latter works are not transitive. 
Partially hyperbolic diffeomorphisms with topological neutral center were also studied recently in \cite{BonattiZhang} where the authors proved that, if $f$ is transitive then there is a continuous system of center arc-lengths preserved by $f$. As a consequence the authors also obtained a topological classification for $f$ when $\text{dim}(M)=3$. 

Also recently, the author, joint with M. E. Noriega and R. Var\~ao \cite{PNV}, studied the disintegration of ergodic invariant measures along an invariant one-dimensional foliation, along which the dynamics preserves a continuous system of arc-lengths. As  a consequence it is showed that for $f:M \to M$ a $C^1$ partially hyperbolic diffeomorphism with one-dimensional topological neutral center, the disintegration of any $f$-invariant ergodic measure $\mu$ along $\mathcal F^c$ is either supported on a countable set, a Cantor set or it is full and the conditional measures along the center leaves are equivalent to the leaf measures given by the arc-length system. 

The main goal of this paper is to study and address some problems on the metric and ergodic properties of partially hyperbolic diffeomorphisms with one-dimensional topological neutral center, in general, assuming that the $f$-invariant measure $\mu$ has full support. In what follows, we address three main lines of investigation which will be detailed in the sequel:
\begin{itemize}
\item the properties of systems of center-metrics preserved by $f$;
\item the occurrence of the Bernoulli property for such systems;
\item an invariance principle concerning the disintegration of an ergodic measure along the center foliation.
\end{itemize}

Naturally, the starting point of all the results obtained here are the works \cite{PNV} and \cite{BonattiZhang}.

\subsection{The Bernoulli property for partially hyperbolic diffeomorphisms with topological neutral center}

Given a measure space $(X,\mathcal B, \mu)$ and a measure preserving automorphism $f:X \to X$ with finite entropy, we say that $f$ is a \textit{Bernoulli automorphism}, that it is a \textit{Bernoulli system} or that it \textit{has the Bernoulli property}, if $(f,\mu)$ is measurably conjugate to a $(\sigma, \rho)$ where $\sigma:\Sigma^{\mathbb Z} \to \Sigma^{\mathbb Z}$, $\Sigma=\{0,1,\ldots, n-1\}$ with  $n\in \mathbb N$, is a standard Bernoulli shift and $\rho$ is the Bernoulli measure in $\Sigma^{\mathbb Z}$ defined by some distribution $p=(p_0, \ldots, p_{n-1})$.

Bernoulli systems are extremely important in ergodic theory and dynamical systems in general due to its huge variety of dynamical and ergodic properties. Although the Bernoulli property is much stronger than mixing, many of the natural examples arising in smooth dynamics which are mixing are actually Bernoulli. For example, Y. Katznelson proved in \cite{K} that every ergodic automorphism of tori is actually a Bernoulli automorphism. Few years latter, it was realized that there was a deep connection between what are called hyperbolic structures and the occurrence of ergodic properties such as ergodicity, mixing and Bernoulli property. In the seminal paper \cite{Anosov1} D. Anosov proved that geodesic flows of negatively curved compact manifolds are ergodic, and furthermore they are \textit{$K$-systems}, i.e, they have completely positive entropy. One of the key properties used in that proof is the fact that the stable and unstable foliations of such dynamical systems are \textit{absolutely continuous}. 


Being a $K$-system is already much stronger than ergodicity and in \cite{OW} using the hyperbolic structure and Ornstein theory, D. Ornstein and B. Weiss proved that geodesic flows in compact surfaces with negative curvature are actually Bernoulli, which is far stronger than ergodicity. The strategy stablished in \cite{OW} was pushed forward by several other authors and for much more general contexts such as: volume-preserving non-uniformly hyperbolic diffeomorphisms \cite{YP3}, non-uniformly hyperbolic singular maps and flows \cite{CH}, partially hyperbolic derived from Anosov diffeomorphisms \cite{PTV2}. In all the cases where uniform or non-uniform hyperbolicity is present, the central roles are played by the absolute continuity of stable and unstable foliations, transversality, the $K$-property and the uniform contraction and expansion of the stable and unstable foliations respectively. For more details on the main arguments used to extend the Kolmogorov to obtain the Bernoulli property we refer the reader to \cite{PVBook}.

We recall that for a partially hyperbolic diffeomorphism $f:M \to M$, defined on a compact Riemannian manifold $M$, we say that 
\begin{itemize}
\item $f$ is \textit{accessible} if any two points $x,y\in M$ may be connected by a concatenation of $C^1$-paths each of which is fully contained in a stable or an unstable leaf of $f$ -- this concatenation is called an $su$-path; 
\item $f$ is \textit{essentially accessible} if any measurable set which is an union of accessibilities classes, must have full of zero volume measure (the accessibility class of a point $x\in M$, $AC(x)$ is the set of all points $y\in M$ which may be reached from $x$ through an $su$-path);
\item $f$ is \textit{center-bunched} if $\nu, \hat{\nu}, \gamma$ and $\hat{\gamma}$ can be chosen so that:
\[\max\{ \nu,\hat{\nu}\} <\gamma \hat{\gamma}.\]
\end{itemize}

Sometimes when working under $C^{1+\alpha}$ regularity, a stronger form of center-bunching is required. We refer the reader to \cite{Aaron1} for more details on it.

All along the paper we denote by $\operatorname{PH}^r_{\mu}(M)$ the set of $C^r$-partially hyperbolic diffeomorphisms on $M$ preserving a given measure $\mu$.

The occurrence of the Bernoulli property for partially hyperbolic diffeomorphisms is a much more delicate issue than the same for the context of (non)uniformly hyperbolic diffeomorphisms, and does not follows from the Kolmogorov property (see a recent example in dimension four given by F. Hertz, A. Kanigowski and K. Vinhage \cite{KHV}). The question of whether the Kolmogorov and the Bernoulli property are equivalent for volume preserving $C^2$ partially hyperbolic diffeomorphisms on three dimensional manifolds is still open. It is worth mentioning here that from \cite{HHU1, ACW}, there exists a $C^1$-open and dense set of Bernoulli diffeomorphisms among the $C^r$, $r>1$, volume preserving partially hyperbolic diffeomorphisms on a compact connected manifold. Moreover, very recently G. N\'u\~nez and J. Hertz \cite{NH} have proved that for a residual set $\mathcal R$ of the family of $C^1$, volume preserving partially hyperbolic diffeomorphisms of a three manifold, given $f\in \mathcal R$ the existence of a minimal expanding or contracting $f$-invariant foliation implies that $f$ is stably Bernoulli. The same authors then conjectured (see \cite[Conjecture 1.2]{NH}) that for a generic set of such diffeomorphisms, either all the Lyapunov exponents vanish almost everywhere or a minimal invariant expanding/contracting foliation exists. 


A strong result by Burns-Wilkinson \cite{BW} states that for $m$ a smooth measure on $M$ and $f \in \operatorname{PH}^2_{m}(M)$, if $f$ is center-bunched and essentially accessible, then $f$ has the Kolmogorov property. This raises the natural question of whether, for these diffeomorphisms, the $K$-property may be pushed to the Bernoulli property. 

\begin{problem} (see Question 11.11, raised by K. Burns in \cite{hasselblatt_2007})
Let $f$ be a $C^{1+\alpha}$ (essentially) accessible, center-bunched, partially hyperbolic diffeomorphism. Is $f$ Bernoulli?
\end{problem}

This question is still widely open and is extremely hard if no other hypothesis is assumed for the center direction.  In the partially hyperbolic situation presented in \cite{PTV2}, that is for partially hyperbolic diffeomorphisms of $\mathbb T^3$ which are homotopic to a linear Anosov, absolute continuity of the center-stable (or center-unstable foliation) is assumed, and the absence of uniform contraction (or uniform expansion) is bypassed by analyzing the measure theoretical behavior of the center foliation and proving that essentially one may reduce each center leaf to a subset where a topological contraction (expansion) occurs and with arbitrarily large density. This approach is only possible because derived from Anosov diffeomorphisms of $\mathbb T^3$ are semi-conjugate to their linearization and, being so, they carry on their central leaves a type of topological contraction (or expansion) over long arcs of center leaves. For a general volume preserving $C^2$ partially hyperbolic diffeomorphism this approach is not possible and, even assuming accessibility, center-bunching condition and existence of an absolutely continuous center-stable foliation, it is not clear how to obtain the Bernoulli property, if this is the case. We also remark that recently, D. Dolgopyat, F. Hertz and A. Kanigownski \cite{DHK} showed that every $C^{1+\alpha}$ conservative diffeomorphism which is exponentially mixing is Bernoulli. This strong result provides new insights on how very strong mixing-type properties may bypass the presence of some non-expanding/non-contracting center behavior.

As the absence of uniform contraction/expansion behavior of the center manifold is a major obstruction to obtain the partially hyperbolic context, it is natural to wonder if some control hypothesis for the center would imply the Bernoulli property. We then address the following problem due to A. Wilkinson.

\begin{problem}(see \cite[Problem 49]{HHUSurvey}) \label{prob:Amie}
Let $f:M \to M$ be a volume preserving $C^{1+\alpha}$-partially hyperbolic diffeomorphism which is accessible and center-bunched. If $f$ has Lyapunov stable center is it true that $f$ is Bernoulli?
\end{problem}


Here we are able to provide a substantial advance to Problem \ref{prob:Amie}, replacing Laypunov stability by topological neutral center and obtaining a dichotomy for the measurable behavior of the center conditionals.

\begin{maintheorem}\label{theorem:1dimens}
Let $f:M\to M$ be a $C^{r+}$, $r\geq 1$, partially hyperbolic diffeomorphism with orientable one dimensional center bundle, whose orientation is preserved by $f$. If $f$ preserves a smooth ergodic measure $\mu$ and is topologically neutral along the center direction, then one of the following holds:
\begin{itemize}
\item the conditional measures of the disintegration of $\mu$ along the center foliation $\mathcal F^c$ are atomic or supported on Cantor subsets of the leaves;
\item the center foliation, $\mathcal F^c$, is leafwise absolutely continuous and $f$ is Bernoulli.
\end{itemize}
\end{maintheorem}

We remark that we are not requiring essential accessibility on Theorem \ref{theorem:1dimens}.
%
To prove Theorem \ref{theorem:1dimens} we first show the trichotomy for the disintegration of $\mu$ along the center manifold (atomicity, Cantor support or leafwise absolute continuity), which will be a consequence of the Theorems presented in the next section, and then we need to show that $f$ is Bernoulli when $\mathcal F^c$ is leafwise absolutely continuous. This second part is proved by revisiting the arguments employed in \cite{YP3, CH, PTV2} and making technical adjustments, more precisely we prove the following.

\begin{maintheorem} \label{theo:KBGeneral}
Let $f:M\rightarrow M$ be a $C^{1+}$ volume preserving partially hyperbolic diffeomorphism which satisfies:
\begin{itemize}
\item[1)] $f$ is dynamically coherent;
\item[2)] $\mathcal F^{cs}$ holonomies between almost every pair of $\mathcal F^{u}$ local leaves are absolutely continuous, or equivalently $\mathcal F^{cs}$ is leafwise absolutely continuous.
\end{itemize}
If $f$ has Lyapunov stable center and is a $K$-automorphism then it is a Bernoulli automorphism.
\end{maintheorem}

\begin{remark}
The equivalence mentioned on the second item  is a consequence of Lemma \ref{lemma:PVW}, which will be proved in Section \ref{sec:mtpf}. Also, since inverse of $K$-systems are also $K$-systems, by taking $f^{-1}$ we may replace $\mathcal F^{cs}$ to $\mathcal F^{cu}$ in the second item.
\end{remark}

For the sake of the reader, we show in Section \ref{sec:appB} how the construction of $\varepsilon$-regular covers can be made using only leafwise absolute continuity of $\mathcal F^c$ instead of absolute continuity, but we conclude the proof of Theorem \ref{theo:KBGeneral} in Appendix A  as the technical adjustments required are not used elsewhere in the paper. 


\subsection{Metric properties of the center foliation}

Given a foliation $\mathcal F$ of $M$ by $C^1$ leaves, for any leaf $L\in \mathcal F$ and $x,y \in L$ we denote by $d^{\mathcal F}(x,y)$ the distance between $x$ and $y$ measured in the Riemannian distance of $L$ (we omit  $L$ in the notation). The distance $d^{\mathcal F}(x,y)$ will be called the \textit{leaf distance} between $x$ and $y$. In case $f:M \to M$ is a $C^1$ partially hyperbolic diffeomorphism with center foliation $\mathcal F^c$, we replace the notation $d^{\mathcal F^c}$ by the more convenient notation $d_c$. That is, for $y\in \mathcal F^c(x)$, $d_c(x,y)$ denotes the distance between $x$ and $y$ measured along the leaf $\mathcal F^c(x)$.

We now define two related but distinct properties: the existence of an invariant arc-length system and leafwise equicontinuity.

The following is a generalization of the concept of center arc-length system defined on \cite{BonattiZhang}.

\begin{definition}\label{Flenghtsystem} (see \cite{BonattiZhang, PNV})
Given a one-dimensional foliation $\mathcal F$ of $M$, invariant by a $C^1$ difeomorphism $f:M\to M$, we will call $\{l_x\}$ a $\mathcal{F}$-arc-length system if, for each $x\in M$,  $l_x$ is a map defined on the simple arcs on $\mathcal{F}(x)$, where two simple arcs are considered the same if one is only a reparametrization of the other,  and  $l_x$ satisfies the following properties:
		\begin{enumerate}
			\item $l_x$ is strictly positive on non-degenerate arcs and vanishes on degenerate arcs,
			\item for any simple arc $\gamma:[0,1]\to \mathcal{F}(x)$ and $a\in (0,1)$, 
			\[
			l_x(\gamma[0,a])+l_x(\gamma[a,1])=l_x(\gamma[0,1]),
			\]
			\item for any simple arc $\gamma:[0,1]\to \mathcal{F}(x)$,
			\[l_x(\gamma[0,1])=l_{f(x)}(f(\gamma[0,1])),\]
		\item[(4)] given a sequence of simple arcs $\gamma_n :[0,1]\to \mathcal F(x_n)$, converging (with respect to the $C^0$ topology) to a simple arc $\gamma:[0,1] \to \mathcal F(x)$, we have
\[l_{x_n}(\gamma_n) \rightarrow l_x(\gamma), \quad \text{as} \; n\rightarrow +\infty.\]	
\end{enumerate}	
	\end{definition}
We will denote by $\text{Diff}^1_{ac}(M)$ the set of all $C^1$ diffeomorphism $f:M \to M$, preserving some one-dimensional foliation $\mathcal F$ and a system of arc-lengths $\{l_x\}_{x\in M}$ over $\mathcal F$. 

 As observed in \cite{PNV} whenever a $f$-invariant foliation is endowed with a $\mathcal{F}$-arc-length system, this arc-length system induces a family of $f$-invariant metrics $\{d_x\}_{x\in M}$ by taking:
 \begin{equation}\label{defi:metrics}
 d_x(y,z):=\min \{l_{x}(\gamma) \, : \, \gamma:[0,1]\to \mathcal{F}(x) \; \text{is simple with} \; \, \gamma(0)=y, \, \gamma(1)=z\}.\end{equation}
This system is also additive in the following sense: given any simple arc $\gamma:[0,1]\to \mathcal{F}(x)$ we have
 \[d_x(\gamma(p_1),\gamma(p_3)) = d_x(\gamma(p_1),\gamma(p_2))+d_x(\gamma(p_2),\gamma(p_3)), \quad \forall \; 0\leq p_1 \leq p_2 \leq p_3 \leq 1.\]
 It is not true that $\{d_x\}_{x\in M}$ is continuous in the global sense, i.e, it is possible that we may find sequences $x_n \rightarrow x$, $y_n \rightarrow y$, with $y_n \in \mathcal F(x_n)$, $y\in \mathcal F(x)$ but $d_{x_n}(x_n,y_n) \nrightarrow d_{x}(x,y)$ (for example, for compact foliations where the leaves do not have uniformly bounded length). As mentioned in \cite{PNV} It is true, however, that restricted to plaques inside local charts this family of metrics are continuous. This property was called \textit{plaque-continuity} in \cite{PNV}, as we recall below.
 
Consider $\mathcal F$ a continuous foliation of $M$. A function $F:\bigcup_{x\in M} \mathcal F(x) \times \mathcal F(x) \to [0,\infty)$ will be called plaque-continuous if given any local chart $U$ of $\mathcal F$, for any sequences  $x_n \rightarrow x$, $y_n \rightarrow y$ with $y_n \in \mathcal F|U(x_n)$, $x\in U$ and $y \in \mathcal F|U(x)$, we have
\[\lim_{n\rightarrow \infty} F(x_n,y_n) = F(x,y).\]
 
 \begin{proposition}\cite{PNV}
 Let $f:M \to M$ be a homeomorphism preserving a one-dimensional continuous foliation $\mathcal F$ endowed with an invariant $\mathcal{F}$-arc-length system. The metric system defined by \eqref{defi:metrics} is plaque-continuous.
 \end{proposition}

\begin{maintheorem}\label{lemma:superaux2}
Let $f\in \text{Diff}^1_{ac}(M)$ with associated foliation $\mathcal F$. Assume that $\mathcal F$ is orientable and $f$ preserves the orientation of $\mathcal F$. Given any plaque-continuous metric system $\{\rho_x\}_{x\in M}$ preserved by  $f$, and $\mu$ a $f$-invariant ergodic measure with full support, then there exists a constant $C>0$ such that
\[\rho_x(x,y) \leq C\cdot  d_x(x,y), \quad \forall x\in M, y\in \mathcal F(x).\]
Moreover, if the disintegration of $\mu$ along $\mathcal F$ is neither atomic nor Cantor, then  $\mathcal F$ is leafwise absolutely continuous with respect to the leaf Lebesgue measure.
\end{maintheorem}

\begin{remark}
We remark that although $\mu$ is assumed to have full support on Theorem \ref{lemma:superaux2}, it does not need to be absolutely continuous with respect to a volume measure on $M$.
\end{remark}

To prove the second part of Theorem \ref{lemma:superaux2} we need to construct an invariant system of metrics which is plaque-continuous and whose distances are somehow comparable to the leaf measures. This is possible in a slightly more general setting where we assume simply that $f$ is equicontinuous along $\mathcal F$.

\begin{definition}
If $f:M\to M$ is a $C^1$ diffeomorphism and $\mathcal F$ is a $f$-invariant foliation, we say that $f$ is equicontinuous along $\mathcal F$ or that $f$ is leafwise equicontinuous (when $\mathcal F$ is implicit), if given $\varepsilon>0$ there exists $\delta>0$ for which, for any pair of points $x,y$ in the same $\mathcal F$-leaf we have
\[d^{\mathcal F}(x,y)<\delta \Rightarrow \sup_{n\in \mathbb Z}d^{\mathcal F}(f^n(x),f^n(y)) < \varepsilon .\]
\end{definition}
We denote by $\text{Diff}^1_{eq}(M)$ the set of all $C^1$ diffeomorphism $f:M \to M$, preserving some one-dimensional foliation $\mathcal F$ along which $f$ is equicontinuous. \\

If $f$ is a $C^1$ diffeomorphism which is equicontinuous along an orientable continuous $f$-invariant foliation $\mathcal F$ with dimension one, we may define a system of invariant metrics over the leaves of $\mathcal F$ by taking 
\[D_x(x,y):=\sup_{n\in \mathbb Z} d^{\mathcal F}(f^n(x),f^n(y)),\] 
for every $y\in \mathcal F(x)$. In what follows we show that this system of metrics is continuous when $\mathcal F$ is orientable and with orientation being preserved by $f$.

\begin{maintheorem} \label{lemma:superaux}
Let $f:M\to M$ be a $C^1$-diffeomorphism which is equicontinuous along an orientable continuous $f$-invariant foliation $\mathcal F$ with dimension one. Assume $f$ preserves the orientation of the leaves. Then, the system of metrics $\{D_x\}_{x\in M}$ given by
\[D_x(x,y):=\sup_{n\in \mathbb Z} d^{\mathcal F}(f^n(x),f^n(y)),\quad \forall \; y\in \mathcal F(x)\]
is continuous. 
 \end{maintheorem}

For the continuous system of metrics given by the previous Theorem, it can be proved that given any ergodic $f$-invariant measure $\mu$, there exists an $f$-invariant subset $S\subset X$ of full $\mu$-measure such that for all $x\in S$ we have
\[\sup_{n\in \mathbb N} d^{\mathcal F}(f^n(x),f^n(y))= D_x(x,y), \quad \forall \; y\in \mathcal F(x).\]
This fact is not essencial to the proof of the main theorems, therefore we prove it in Appendix B..



\subsection{Local invariance principle for the center disintegration}

Consider a subset $U \subset M$ 
foliated by a pair of continuous transversal foliations $\mathcal F$ and $\mathcal G$ with respect to which $U$ is a product set, that is, $U$ has global product structure with respect to $\mathcal F$ and $\mathcal G$ in the sense that if we define
\begin{equation}\label{eq:Q}
Q: \mathcal F(x_0) \times \mathcal G(x_0) \to M, \quad Q(a,b):=\mathcal G(a) \cap \mathcal F(b),\end{equation}
then
\[U = Q(A^{\mathcal F} \times B^{\mathcal G}),\]
for some subsets $A^{\mathcal F} \subset \mathcal F(x_0)$,  $B^{\mathcal G}\subset \mathcal G(x_0)$. Inside $U$ we may define global $\mathcal F$-holonomies between two $\mathcal G$-leaves and vice versa. Given $x,y\in U$, we define
\[H^{\mathcal F}_{x,y}: \mathcal G(x) \to \mathcal G(y) , \quad H^{\mathcal F}_{x,y}(z):= \mathcal F(z)\cap \mathcal G(y).\]
Since $\mathcal F$ is continuous, $H^{\mathcal F}_{x,y}$ is a homeomorphism.

Denote by $\{\mu^{\mathcal F}_x\}$ the disintegration of $\mu$ along the plaques of $\mathcal F$ on $U$ and $\{\mu^{\mathcal G}_x\}$ the disintegration of $\mu$ along the plaques of $\mathcal G$ on $U$. We say that the disintegration $\{\mu^{\mathcal G}_x\}$ is \textit{invariant by $\mathcal F$-holonomies} \footnote{We remark that in \cite{TahzibiYang} the definition is slightly different as the the invariance of the conditional measures is required to hold only inside a full measure subset of $M$.} if
\[\forall \; x\in U, y\in \mathcal F(x), \quad  (H^{\mathcal F}_{x,y})_*\mu^{\mathcal G}_x = \mu^{\mathcal G}_y.\]

The following Lemma proved in \cite{YangTahzibi} shows that when the disintegration along $\mathcal G$ is invariant by $\mathcal F$-holonomies the measure has local product structure.

\begin{lemma}\label{lemma:YT1} (see \cite[Lemma 4.2]{YangTahzibi})
If $\{\mu^{\mathcal F}_x\}$ is $\mathcal G$-invariant then $\{\mu^{\mathcal G}_x\}$ is $\mathcal F$-invariant and $\mu = Q_*(\mu_{x_0}^{\mathcal F} \times \mu_{x_0}^{\mathcal G})$ for typical $x_0\in U$.
\end{lemma}

For the sake of simplicity we fix the following nomenclature: if a certain disintegration is invariant by $\mathcal F^u$-holonomies, where $\mathcal F^u$ is the unstable foliations of a certain partially hyperbolic map, we say that it is \textit{$u$-invariant} (we define \textit{$s$-invariance} in analogy to this definition).

In several settings, asymptotic properties of certain partially hyperbolic dynamics imply the existence of center disintegrations which are invariant by stable and unstable holonomies. In general, the occurrence of this phenomenon provides some rigidity for the system in terms of a conjugacy with a simpler model, see for example \cite{AV, AVW, AVWII, TahzibiYang}.

In our context we show that, if the support of the center conditionals is full, then restricted to local charts there is a continuous disintegration of $\mu$ along the center foliation which is locally invariant by stable and unstable holonomies.

\begin{maintheorem}\label{theorem:IP}
Let $f:M\to M$ be a $C^{r+}$, $r\geq 1$, partially hyperbolic diffeomorphism with orientable one dimensional topological neutral center bundle, whose orientation is preserved by $f$. Let $\mu$ be an ergodic $f$-invariant probability measure with full support. Then there is a finite cover of $M$ by open neighborhoods $\mathcal U$, such that for each $U\in \mathcal U$ either:
\begin{itemize}
\item[1)] the conditional measures of the disintegration of $\mu( \cdot | U)$ along $\mathcal F|U$ are atomic or supported on Cantor subsets of the respective leaves or
\item[2)] there is a disintegration of $\mu( \cdot | U)$ along $\mathcal F^c|U$ which is continuous, $s$-invariant and $u$-invariant.
\end{itemize}
\end{maintheorem}

\begin{remark}
In the previous theorem we use the notation $\mu(\cdot | U)$ to denote the restriction of $\mu$ to $U$, that is, it is the probability measure on $U$ given by $\mu(B | U):=\mu(U)^{-1}\cdot \mu(B)$.  The notation $\mathcal F|U$ stands for the foliation on $U$ induced by the restriction of $\mathcal F$ on $U$.
\end{remark}

As a corollary we show that for such conservative diffeomorphisms, $\mathcal F^{c}$ is leafwise absolutely continuous if, and only if, either (and consequently both) $\mathcal F^{cs}$ or $\mathcal F^{cu}$ are leafwise absolutely continuous.

\begin{maincorollary}\label{corofn}
Let $f:M\to M$ be a $C^{r+}$, $r\geq 1$, conservative partially hyperbolic diffeomorphism with orientable one dimensional topological neutral center bundle. The following are equivalent:
\begin{itemize}
\item[1)] $\mathcal F^c$ is leafwise absolutely continuous.
\item[2)] $\mathcal F^{cs}$ is leafwise absolutely continuous.
\item[3)] $\mathcal F^{cu}$ is leafwise absolutely continuous.
\end{itemize}
\end{maincorollary}
\begin{proof}
Let $\mu$ be the smooth measure preserved by $f$. From \cite{AVWII} we already know that (1) implies\footnote{This does not depend on the topological neutral hypothesis.} (2) and (3). Assume (2) holds. Therefore, for a typical center stable leaf $L^{cs}$ we have 
\[\mu_{L^{cs}} \sim \lambda_{L^{cs}},\]
where $\mu_{L^{cs}}$ denotes the conditional measure of $\mu$ along the center-stable leaf $L^{cs}$ and $\lambda_{L^{cs}}$ is the leaf Lebesgue measure of $L^{cs}$.
Now, by the previous theorem and by \cite{TahzibiYang}, we have that, inside $L^{cs}$, the center-holonomies preserve the disintegration of $\mu_{L^{cs}}$ along the unstable, which are equivalent to Lebesgue. That is, $\mathcal F^c$ is leafwise absolutely continuous inside $L^{cs}$. Since this holds for almost every center-stable leaf, it follows that $\mathcal F^c$ is leafwise absolutely continuous on $M$ as a whole. Thus (2) implies (1). Analogously we prove that (3) implies (1), concluding the proof.
\end{proof}

The conclusion of Corollary \ref{corofn} is not true in general, even if $\mathcal F^{cs}$ is smooth. Indeed in \cite{PT} the authors construct partially hyperbolic maps, homotopic to a linear Anosov map on $\mathbb T^3$, with smooth center-unstable foliation and whose center Lyapunov exponent is zero for Lebesgue almost every point. Recently A. Tahzibi and J. Zhang \cite{TahzibiZhang} proved that the center foliation for these diffeomorphisms must be atomic with respect to the volume measure (as, in this case, it is a non-hyperbolic invariant measure). Therefore $\mathcal F^{cs}$ is smooth but $\mathcal F^c$ is not leafwise absolutely continuous.

At last, using Theorem \ref{theorem:IP} we are also able to show that in the conservative case, if $f$ is locally accessible (see definition below) and the center conditionals have full support, then the center foliation is as regular as $f$. 

\begin{definition} (cf. \cite[Definition 2.1]{KatokKononenko})
We say that a partially hyperbolic diffeomorphism $f:M \to M$ is locally accessible if given local chart $U$ of $M$ and any $x,y\in M$, there exists a sequence $x_0=x,x_1,x_2, \ldots, x_{n-1}, x_n =y$ with
\begin{itemize}
\item[1)] $x_i \in U$ for all $0\leq i \leq n$,
\item[2)] $x_i \in \mathcal F^{\tau_i}|U(x_{i-1})$, $1\leq i \leq n$, where $\tau_i \in \{s,u\}$.
\end{itemize}
\end{definition}

\begin{maintheorem}\label{theorem:centerregular}
Let $f:M\to M$ be a $C^{r+}$, $r\geq 1$, locally accessible partially hyperbolic diffeomorphism with orientable one dimensional topological neutral center bundle, whose orientation is preserved by $f$. Let $\mu$ be an ergodic smooth $f$-invariant probability measure. Either:
\begin{itemize}
\item[1)] the conditional measures of the disintegration of $\mu( \cdot | U)$ along $\mathcal F|U$ are atomic or supported on Cantor subsets of the respective leaves or
\item[2)] $\mathcal F^c$ is a $C^r$ foliation.
\end{itemize}
\end{maintheorem}

Local accessibility is clearly stronger than accessibility and has been verified only for certain very selective classes of partially hyperbolic dynamics. A natural question is:

\begin{problem}
 Can we replace local accessibility by accessibility on the hypothesis of Theorem \ref{theorem:centerregular} ?
 \end{problem}

\section{Preliminaries on foliations and measure theory}\label{sec:mtpf}



Let $M$ be a manifold of dimension $d\geq 2$. A foliation with $C^r$ leaves, $r\geq 1$, is a partition $\mathcal F$ of $M$ into $C^r$ submanifolds of dimension $k$, for some $0<k<d$ and $1\leq r\leq \infty$, such that for every $p\in M$ there exists a continuous \textit{local chart}
\[\Phi:B^k_2 \times B_2^{d-k} \to M \quad (B^m_2 \text{ denotes the ball of radius $2$ in } \; \mathbb R^m)\]
with $\phi(0,0)=p$ and such that the restriction to every horizontal $B_2^k \times \{\eta\}$ is a $C^r$ embedding depending continuously on $\eta$ and whose image is contained in some $\mathcal F$-leaf. The image $\mathfrak B = \Phi(B^k_2 \times B_2^{d-k}) $ is called a \textit{foliation box} and the sets $\Phi(B_2^k \times \{\eta\})$ are called \textit{local leaves} or \textit{plaques} of $\mathcal F$ in the given foliation box.  For any $\xi \in B_2^k$, the set $\mathcal T= \phi(\{\xi\} \times B_2^{d-k})$ is called a \textit{local transversal} to $\mathcal F$. 
The restriction of a local chart $\Phi:B^k_2 \times B_2^{d-k} \to M$ to $\overline{B^k_1} \times \overline{B_1^{d-k}}$ is called a \textit{closed local chart} and the image $\mathfrak C = \Phi(\overline{B^k_1} \times \overline{B_1^{d-k}}) $ is called a \textit{closed foliation box}.

Given a subset $T \subset M$ we say that $T$ is \textit{transversal} to $\mathcal F$ if for every $x\in T$, there exists a foliation box $\mathfrak B$ containing $x$ for which the connected component of $T\cap \mathfrak B$ containing $x$ is a local transversal to $\mathcal F$. 

Along the paper, given a manifold $N$ we will use the notation $\lambda_N$ to denote the volume measure on $N$ induced by its Riemannian structure. We sometimes refer to this measure as being the \textit{Lebesgue measure} of $N$.

\begin{definition}
Given a foliation $\mathcal F$ of $M$ by $C^r$-leaves and $\mathcal T_1$ and $\mathcal T_2$ two local transversals inside a foliation box $\mathfrak B$, the local $\mathcal F$-holonomy between $\mathcal T_1$ and $\mathcal T_2$ is the map $h_{\mathcal T_1,\mathcal T_2}:\mathcal T_1 \to \mathcal T_2$ given by
\[h_{\mathcal T_1,\mathcal T_2}(x) = \Phi(B_1^k \times \{\eta\}) \cap \mathcal T_2, \]
where $\eta = \pi_2 \circ \phi^{-1}(x)$.

Given $T_1, T_2 \subset M$ transversals to $\mathcal F$, for $x\in T_1, y\in T_2$ we say that $h_{x,y}: U_1 \to U_2$ is a $\mathcal F$-holonomy map if 
\begin{itemize}
\item $U_1 \subset T_1$ is a neighborhood of $x$ in $T_1$, $U_2\subset T_2$ is a neighborhood of $x$ in $T_2$;
\item there exists a foliation box $\mathfrak B$ such that $U_1$ and $U_2$ are local transversals in $\mathfrak B$;
\item $h_{x,y}$ is the restriction to $U_1$ of a local $\mathcal F$-holonomy .
\end{itemize}
\end{definition}

%
%

\begin{definition}
We say that a foliation $\mathcal F$ is absolutely continuous if given any pair of local smooth transversals $T_1$ and $T_2$ the holonomy map $h_{T_1,T_2}$ defined by $\mathcal F$ between $T_1$ and $T_2$ is absolutely continuous with respect to the Riemannian measures $\lambda_{T_1}$ and $\lambda_{T_2}$ defined in $T_1$ and $T_2$ respectively.
\end{definition}

Absolute continuity of a foliation is a measure theoretical property which implies, in a certain sense, a version of the Fubini theorem for the foliation.  Let $(X, \mu, \mathcal B)$ where $X$ is a polish metric space, $\mu$ a finite Borel measure on $X$ and $\mathcal B$ the Borel $\sigma$-algebra of $X$. For a partition $\mathcal P$ of $X$ by measurable sets, considering the projection $\pi:X \rightarrow \mathcal P$ we may define the measure space $(\mathcal P, \widehat \mu, \widehat{\mathcal B})$ where $\widehat \mu := \pi_* \mu$
and  $\widehat B \in \widehat{\mathcal B}$ if and only if $\pi^{-1}(\widehat B) \in \mathcal B$.

 Given a partition $\mathcal P$. A family of measures $\{\mu_P\}_{P \in \mathcal P}$ is called a \textit{system of conditional measures} for $\mu$ along $\mathcal P$ if
\begin{itemize}
 \item[i)] for every continuous function $\phi:X \to \mathbb R$ the map $P \mapsto \int \phi \; d\mu_P$ is measurable;
\item[ii)] $\mu_P(P)=1$ for $\widehat \mu$-almost every $P\in \mathcal P$;
\item[iii)] for every continuous function $\phi:X \to \mathbb R$,
 \[\displaystyle{ \int_M \phi \; d\mu = \int_{\mathcal P}\left(\int_P \phi \; d\mu_P \right) d\widehat \mu }.\]
\end{itemize}

If $\{\mu_P\}_{P \in \mathcal P}$ is a system of conditional measures for $\mu$ along $\mathcal P$ we also say that the family $\{\mu_P\}$ \textit{disintegrates} the measure $\mu$ or that it is the \textit{disintegration of $\mu$ along $\mathcal P$}.  

It is a well known fact (see  \cite{EW, Ro52}) that when the disintegration of $\mu$ with respect to a partition $\mathcal P$ exists then it is essentially unique. The disintegration of a measure along a partition does not always exists. We say that a partition $\mathcal P$ is a \textit{measurable partition} (or \textit{countably generated}) with respect to $\mu$ if there exist a family of measurable sets $\{A_i\}_{i \in \mathbb N}$ and a measurable set $F$ of full measure such that 
if $B \in \mathcal P$, then there exists a sequence $\{B_i\}$, where $B_i \in \{A_i, A_i^c \}$ such that $B \cap F = \bigcap_i B_i \cap F$.
For measurable partitions $\mathcal P$ of Polish metric spaces endowed with a finite Borel probability measure $\mu$, there is always a disintegration of $\mu$ along $\mathcal P$ \cite{Ro52}.




\begin{definition}
We say that a foliation $\mathcal F$ is leafwise absolutely continuous, or that volume has Lebesgue disintegration along $\mathcal F$-leaves, if for almost every leaf $L$, the conditional measure $m_L$ of $m$ along the leaf is equivalent to the measure $\lambda_L$ on the leaf. 
\end{definition}

It is a classical fact that absolute continuity implies Lebesgue disintegration of volume (see \cite[Lemma $3.4$]{AVW}) but the opposite is not true.

To prove the next proposition we use a lemma due to Pugh-Viana-Wilkinson.

\begin{lemma}[Pugh-Viana-Wilkinson, \cite{PVW}] \footnote{In \cite{PVW} the hypothesis on $\mathcal T$ is actually that it is a local transverse absolutely continuous foliation. However it is easy to see from their proof that it is enough to assume that $\mathcal T$-holonomies between $\mathcal F$-leaves are absolutely continuous.} \label{lemma:PVW}
If volume has Lebesgue disintegration along a foliation $\mathcal F$, then for every 
transverse local foliation $\mathcal T$ to $\mathcal F$ with the property that $\mathcal T$-holonomies between $\mathcal F$ leaves are absolutely continuous, the local $\mathcal F$-holonomy map $h_{\mathcal F}$ between $m$-almost every pair of $\mathcal T$-leaves is absolutely continuous in the sense that given any local leaf $L_0$ of $\mathcal F$, for $\lambda_{L_0} \times \lambda_{L_0}$ -almost every pair $(x,x')\in L_0\times L_0$ the local $\mathcal F$-holonomy between $\mathcal T(x)$ and $\mathcal T(x')$ is absolutely continuous.
\end{lemma}

\begin{corollary} \label{cor:AuxtoPVW}
Let $\mathcal F$ be a foliation for which volume has Lebesgue disintegration and $\mathcal T$ be an absolutely continuous transversal foliation to $\mathcal F$. Denote by $\{m^{\mathcal T}_x\}_x$ the disintegration of the volume measure $m$ along $\mathcal T$ and $\nu_x$ the factor measure induced on $\mathcal F(x)$. Then, for almost every $x$ and for $\nu_x$-almost every $y\in \mathcal F(x)$ the $\mathcal F$-holonomy map between $\mathcal T(x)$ and $\mathcal T(x')$ is absolutely continuous.
\end{corollary}
\begin{proof}

Take $L_0=\mathcal F(z)$ arbitrarily. By Lemma \ref{lemma:PVW}  we may take $x\in L_0$ and $R\subset L_0$ such that $\lambda_{L_0}(R)$ has full measure in $L_0$ and for every $y\in R$ the holonomy between $\mathcal T(x)$ and $\mathcal T(y)$ is absolutely continuous. Since $\mathcal T$ is absolutely continuous then for every $\mathcal F$-leaf, $\mathcal F(z')$ we have that $h^{\mathcal T}_{z,z'}(R)$ also has full $\lambda_{\mathcal F(z')}$-measure. In particular, since $\mathcal F$ is leafwise absolutely continuous, the set 
\[\mathcal T(R):= \bigcup_{y\in R}\mathcal T(y)\]
has full $m$-measure.
Now, for the initial $x \in L_0$ fixed, we know that
\[m(\mathcal T(R)) =\nu_x(R) \Rightarrow \nu_x(R)=1.\]
As $x$ can be chosen inside a full $\lambda_{L_0}$-measure inside each central leaf $L_0$, by the leafwise absolute continuity of $\mathcal F$ it follows that, for almost every $x$ and for $\nu_x$-almost every $y\in \mathcal F(x)$ the $\mathcal F$-holonomy map between $\mathcal T(x)$ and $\mathcal T(x')$ is absolutely continuous as we wanted to show. 
\end{proof}

In \cite{PNV} the authors address the problem of determining how the existence of an invariant arc-length system, over a certain one-dimensional foliation $\mathcal F$, impose restrictions on the conditional measures given by a certain invariant ergodic measure. The main result is that there are only three types of possibilities for the conditional measures, which we recall below.

\begin{theorem} \cite[Theorem A]{PNV} \label{PNV1}
	Let $f:M\to M$ be a  homeomorphism over a compact smooth manifold $M$,  $\mathcal F$  be a $f$-invariant one-dimensional continuous foliation of $M$ by $C^1$-submanifolds and $\{l_x\}$ a continuous $\mathcal{F}$-arc length system. If $f$ is ergodic with respect to a $f$-invariant measure $\mu$ then one of the following holds:
		\begin{itemize}
			\item[a)] the disintegration of $\mu$ along  $\mathcal F$ is atomic.
			\item[b)] for almost every $x\in M$, the conditional measure on $\mathcal F(x)$ is equivalent to the measure $\lambda_{x}$ defined on simple arcs of $\mathcal F(x)$ by: \[\lambda_{x}(\gamma([0,1]))=l_{x}(\gamma), \; \text{where} \; \gamma \; \text{is a simple arc.}\]
			\item[c)] for almost every $x\in M$, the conditional measure on $\mathcal F(x)$ is supported in a Cantor subset of $\mathcal F(x)$.
		\end{itemize}
\end{theorem}

As remarked in \cite{PNV}, Theorem \ref{PNV1} applies directly to transitive $C^1$ partially hyperbolic diffeomorphism with one-dimensional topological neutral center direction yielding the following.

	\begin{theorem} \cite[Theorem B]{PNV} \label{PNV}
			Let $f:M \to M$ be a transitive $C^1$ partially hyperbolic diffeomorphism with one-dimensional topological neutral center direction. If $f$ is ergodic with respect to a $f$-invariant measure $\mu$ then one of the following holds:
		\begin{itemize}
			\item[a)] the disintegration of $\mu$ along  $\mathcal F^c$ is atomic.
			\item[b)] for almost every $x\in M$, the conditional measure on $\mathcal F^c(x)$ is equivalent to the measure $\lambda_{x}$ defined on simple arcs of $\mathcal F^c(x)$ by: \[\lambda_{x}(\gamma([0,1]))=l_{x}(\gamma), \; \text{where} \; \gamma \; \text{is a simple arc.}\]
			\item[c)] for almost every $x\in M$, the conditional measure on $\mathcal F^c(x)$ is supported in a Cantor subset of $\mathcal F^c(x)$.
		\end{itemize}	
		\end{theorem}

\section{ Proof of Theorem \ref{lemma:superaux}} \label{sec:onedimensionalcase}

\begin{proof}[Proof of Theorem \ref{lemma:superaux}]
Let us first prove the plaque-continuity of the system of metrics. Considering the order relation along a leaf $L\in \mathcal F$ induced by the orientation of $\mathcal F$, for each $x\in M$ denote
\[x^r:= \sup \{y \in \mathcal F(x): D_x(x,y) \leq r\} , \quad r\geq 0,\]
\[x^r:= \inf \{y \in \mathcal F(x): D_x(x,y) \leq -r\} , \quad r< 0.\]
In particular, if $\mathcal F(x)$ is a leaf of diameter less than $r$ then $x^r  = x^s$, $\forall r\leq s$. 
The point $x^r$ is well defined by the equicontinuity of $f$ along $\mathcal F$. 
We also consider $\phi:\mathbb R \times M \to M$ the flow along the center foliation induced by $d^{\mathcal F}$ and the orientation of $\mathcal F$. For $s \in \mathbb R$ and $y=\phi_s(x)$,  denote 
\[[x,y] = \phi([0,s] \times \{x\}), \quad \text{if} \quad  s\geq 0 \quad \text{and} \quad [x,y] = \phi([s,0] \times \{x\}), \quad \text{ otherwise. }\]

\noindent \textbf{Claim.} For each $r\in \mathbb R$ fixed, the map $g:M \to \mathbb R$ given by $x\mapsto x^r$ is continuous.
\begin{proof}[proof of Claim.] 

Let $x\in M$ and consider $(n_k) \subset \mathbb Z$ such that
\[r-k^{-1} < d^{\mathcal F}(f^{n_k}(x), f^{n_k}(x^r)) \leq r, \quad \forall k\in \mathbb N.\]
Observe that the sequence $(n_k)$ may be constant.

Consider $\mathfrak C$ be a closed foliation box associated to a closed local chart $\psi:  [0,1]^m \to \mathfrak C$ such that $\psi(\{1/2\}^{d-1} \times [0,1]) = [x,x^r]$, with $\psi( \{1/2\}^{d-1}\times\{0\})=x$ and $\psi(\{1,2\}^{d-1}\times \{1\})=x^r$.

Let $(y_m)\subset \mathfrak C$ with $y_m \rightarrow x$ and let $\widetilde{y_m} := \psi( (\pi_{d-1}\circ \psi^{-1}(y_m)) \times \{1\} )$, where $\pi_{d-1}: [0,1]^d \to [0,1]^{d-1}$ is the projection onto the $d-1$ first coordinates. In other words, $\widetilde{y_m}$ is the intersection of the plaque of $y_m$ in $\mathfrak C$ with the upper cap $\psi([0,1]^{d-1}\times \{1\})$ as showed in figure \ref{fig:ypsilon}. 

\begin{figure}[htb!] \label{fig:ypsilon}
\begin{center}
\includegraphics[height=5cm]{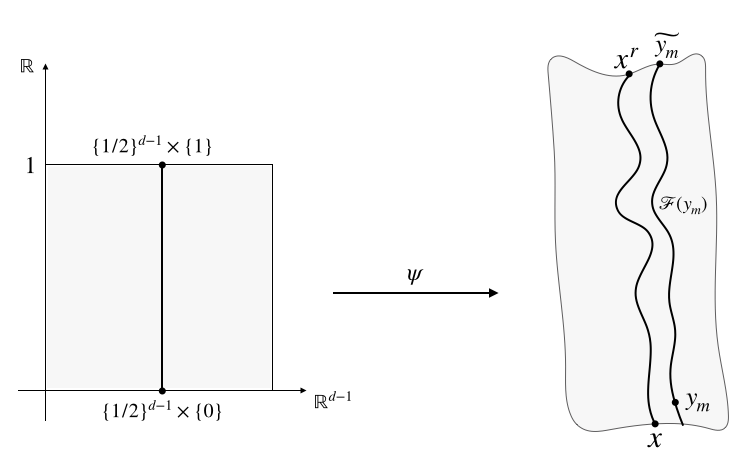}
\caption{$\widetilde{y_m}$ is the intersection of the plaque of $y_m$ in $\mathfrak C$ with the upper cap $\psi([0,1]^{d-1}\times \{1\})$.}
\end{center}
\end{figure}

It is clear that $\widetilde{y_m} \rightarrow x^r$ by the continuity of $\mathcal F$. In particular, for each $k\in \mathbb N$, by the continuity of $f^{n_k}$ and of $\mathcal F$, there exists $m_k$ for which:
\[d^{\mathcal F}(f^{n_k}(x), f^{n_k}(x^r)) - k^{-1} < d^{\mathcal F}(f^{n_k}(y_{m_k}), f^{n_k}(\widetilde{y_{m_k}})) < d^{\mathcal F}(f^{n_k}(x), f^{n_k}(x^r)) + k^{-1} \Rightarrow \]
\begin{equation}\label{eq:mk}
r-2k^{-1} < d^{\mathcal F}(f^{n_k}(y_{m_k}), f^{n_k}(\widetilde{y_{m_k}})) < r+k^{-1}.\end{equation}

Assume without loss of generality that $\widetilde{y_{m_k}}$ is between $y_{m_k}$ and $y^r_{m_k}$, the other case is analogous. We claim that given any $\varepsilon>0$ we can take $k$ large enough so that 
\begin{equation}\label{eq:epsilon2}
d^{\mathcal F}(\widetilde{y_{m_k}}, y^r_{m_k}) < \frac{\varepsilon}{2}.\end{equation}
Indeed, assume \eqref{eq:epsilon2} is false. Then, for a certain $\varepsilon>0$ we have $d^{\mathcal F}(\widetilde{y_{m_k}}, y^r_{m_k}) \geq \varepsilon/2$ for all $k \in \mathbb N$. By the topological neutral center property of $f$, there exists $\delta>0$ for which
\[d^{\mathcal F}(f^j(\widetilde{y_{m_k}}), f^j(y^r_{m_k})) \geq \delta, \quad \forall j\in \mathbb Z.\]
But then,
\begin{align*}
d^{\mathcal F}(f^j(y_{m_k}), f^j(y^r_{m_k})) &= d^{\mathcal F}(f^j(y_{m_k}), f^j(\widetilde{y_{m_k}}))+d^{\mathcal F}(f^j(\widetilde{y_{m_k}}), f^j(y^r_{m_k})) \\
& \geq \delta +  d^{\mathcal F}(f^j(y_{m_k}), f^j(\widetilde{y_{m_k}})), \quad \forall j\in \mathbb Z.\end{align*}
In particular by \eqref{eq:mk} we have
\[r\geq d^{\mathcal F}(f^{m_k}(y_{m_k}), f^{m_k}(y^r_{m_k})) > \delta + r-2k^{-1},\]
which yields an absurd when we take $k \rightarrow \infty$. Therefore \eqref{eq:epsilon2} holds. Now, since $\widetilde{y_{m_k}} \rightarrow x^r$, by \eqref{eq:epsilon2} we conclude that for $k$ large enough we have $d(x^r, y^r_{m_k}) < \varepsilon$. That is, the map $x \mapsto x^r$ is continuous. 
\end{proof}

 Consequently, since $\mathcal F$ is a continuous foliation we have that $g(x):=d^{\mathcal F}(x,x^r)$ is plaque-continuous for every $r\in \mathbb R$ fixed. Since $f$ is equicontinuous along $\mathcal F$, it is not difficult to see that, for each $x\in M$ fixed, the map $r \mapsto x^r$ is also continuous. Therefore $r\mapsto d^{\mathcal F}(x,x^r)$ is plaque-continuous for each $x\in M$.

 Now, inside a local chart $U$,  let $x_n \rightarrow x$ and $y_n \rightarrow y$ with $x_n,y_n$ in the same plaque and $x,y$ in the same plaque. For each $n\in \mathbb N$ we may write $y_n = x_n^{r_n}$, with $r_n$ being the infimum of such possible values. By the definition of $D_x$ it is clear that
 \[\liminf D_{x_n}(x_n,y_n) \geq D_x(x,y).\]
 Now assume that $\limsup D_{x_n}(x_n,y_n) > D_x(x,y)$. Then, by taking a subsequence if necessary we may assume that $\lim_{n\rightarrow \infty} D_{x_n}(x_n,y_n) \geq D_x(x,y) +\delta_2$, for some $\delta_2>0$. Thus, there exists a sequence $l_n$ with 
 \[d^{\mathcal F}(f^{l_n}(x_n),f^{l_n}(y_n)) \geq r +\delta_2/2,\]
 where $r:=D_x(x,y)$. In particular, the point $x_n^r$ must be in $[x_n , y_n]$ and
  \[r+ d^{\mathcal F}(f^{l_n}(x^r_n),f^{l_n}(y_n)) \geq  d^{\mathcal F}(f^{l_n}(x_n),f^{l_n}(x_n^r)) + d^{\mathcal F}(f^{l_n}(x^r_n),f^{l_n}(y_n)) = d^{\mathcal F}(f^{l_n}(x_n),f^{l_n}(y_n)) \geq r +\delta_2/2 \]
  \[\Rightarrow d^{\mathcal F}(f^{l_n}(x^r_n),f^{l_n}(y_n)) \geq \delta_2/2.\]
 By equicontinuity for some $\delta_3>0$ we have
 \[d^{\mathcal F}(f^{l}(x_n^r),f^{l}(y_n)) \geq \delta_3, \quad \forall \; l \in \mathbb Z, \quad \forall n\geq n_0.\]
 Let $\varepsilon_3>0$ be such that,
 \[x\in M, y\in \mathcal F(x), \; d^{\mathcal F}(x,y)<\varepsilon_3 \Rightarrow d^{\mathcal F}(f^k(x),f^k(y)) < \delta_3, \; \forall k \in \mathbb Z.\]
 Since $y_n\rightarrow y$ and $x_n^r \rightarrow x^r=y$, by continuity of $w\mapsto w^r$, there exists $k_0\in \mathbb Z$ such that $k\geq k_0$ implies 
 \[d(y_k,x_k^r) < \varepsilon_3 \Rightarrow d(f^j(y_k),f^j(x_k^r)) < \delta_3, \quad \forall j,\]
 which yields an absurd. That is, we have proved that
 \[\limsup D_{x_n}(x_n,y_n) \leq D_x(x,y) \leq \liminf D_{x_n}(x_n,y_n),\]
 from where it follows that $\lim_{n\rightarrow \infty} D_{x_n}(x_n,y_n) = D_x(x,y) $ as we wanted to show.

%
\end{proof}



\section{Proof of Theorem \ref{lemma:superaux2}}

\begin{proof}[Proof of Theorem \ref{lemma:superaux2}]
Consider $\mathcal U$ a finite cover of $M$ by open charts of $\mathcal F$ such that, restricted to any $U\in \mathcal U$, for $x\in U$ the maps $\rho(x,\cdot )$ and $d_x(x,\cdot)$ are continuous on the plaque $\mathcal F|U(x)$.

For any $x\in M$, there exists $r_x>0$ small enough, so that the map $c(x,r)$ given by
\[c(x,r):= \frac{\rho_x(x_r^-,x_r^+)}{2r}, \quad \text{where} \; (x_r^-,x_r^+):=B_{d_x}(x,r),\]
is well defined for $r< r_x$. More precisely, $r_x$ is given by
\[r_x = \min_{U\in \mathcal U} \{d_x(x,\partial U) : x\in U\; \}.\]
By plaque continuity of the metric system it follows that $x\mapsto r_x$ is continuous, therefore uniformly bounded away from zero, say $0<r_0 \leq r_x$ for all $x\in M$. We now consider the restriction $c: M\times [0,r_0) \to (0,\infty)$. Observe that $(x,r) \mapsto c(x,r)$ is continuous and $f$-invariant on $x$, that is, $c(f(x),r)=c(x,r)$. In particular, for each $r>0$ there exists $c(r)>0$ and a set of full measure $\mathcal P_r$ with,
\[x\in \mathcal P_r \Rightarrow c(x,r)=c(r).\]
By continuity of $c(x,r)$ and the fact that $\mu$ has full support, it follows that $\mathcal P_r$ is dense on $M$, thus $c(x,r)=c(r)$ for all $x\in M$. By continuity of $c(x,r)$, it follows that $c(r)$ is continuous on $(0,+\infty)$, in particular, it is upper bounded on $[r_0/4,r_0/2]$ by a constant $C>0$.

Now, for any $r_0/2>r>0$, since $d_x$ is additive, given any $x\in M$ we may write
\[B_{d_x}(x,2r) = (x_{2r}^-,x]\cup(x,x_{2r}^-) = B_{d_x}(y_1,r)\cup \{x\} \cup B_{d_x}(y_2,r), \]
for certain $y_1 \in  (x_{2r}^-,x)$, $y_2 \in (x,x_{2r}^-) $.
In particular,
\[c(2r)= \frac{\rho_x(x_{2r}^-,x_{2r}^+)}{4r} \leq  \frac{\rho_x(x_{2r}^-,x)}{4r} + \frac{\rho_x(x,x_{2r}^-)}{4r} =\frac{1}{2}c(r) + \frac{1}{2}c(r) = c(r).\]
Thus, 
\begin{equation}\label{eq:supremumC}
\sup_{r\in (0,r_0/2]}c(r) \leq \sup_{r\in [r_0/4, \;r_0/2]}c(r) \leq C.\end{equation}

Given any $x \in M$ and $y\in \mathcal F(x)$, we may take $x_0=x < x_1 < x_2 < \ldots < x_n=y$ such that $d_x(x_i,x_{i+1}) < r_0/2$ for all $0\leq i \leq n-1$ and with $x_i, x_{i+1} \in U_i$ for some $U_i\in \mathcal U$. Thus, by additivity of $d_x$ and \eqref{eq:supremumC} we have
\[\rho_x(x,y) \leq \sum_{i=0}^{n-1} \rho_x(x_i,x_{i+1}) \leq \sum_{i=0}^{n-1} C\cdot d_x(x_i,x_{i+1}) = C\cdot d_x(x,y), \]
which proves the first part.

For the second part, consider the continuous system of distances $\{D_x\}_{x\in M}$ constructed on Theorem \ref{lemma:superaux}. By what we have proved above, there exists $C>0$ such that
\begin{equation} \label{eq:metricx}
D_x(x,y) \leq C\cdot d_x(x,y) \Rightarrow d^{\mathcal F}(x,y) \leq C\cdot d_x(x,y), \quad \forall \; x\in M, y\in \mathcal F(x).\end{equation}
If, for $\mu$ almost every $x\in M$ we have $\mu_x \sim\lambda_x$ then if $\lambda_x(E)=0$, for every $\varepsilon>0$ there exists a cover of $E$ by open $d_x$-balls $E\subset \bigcup_{j \in \mathbb N}B_{d_x}(a_j,r_j)$,
with $\sum_{j\in \mathbb N} r_j < C^{-1}\cdot \varepsilon$. But by \eqref{eq:metricx} we have $B_{d_x}(a_j,r_j) \subset B_{d^{\mathcal F}}(a_j,C\cdot r_j)$, which implies that
\[E\subset \bigcup_{j \in \mathbb N}B_{d^{\mathcal F}}(a_j,C\cdot r_j) \quad \text{with} \quad \sum_{j\in \mathbb N} C\cdot r_j < \varepsilon.\]
In particular $\lambda_{\mathcal F(x)}(E) =0$ as we wanted. Therefore $\lambda_{\mathcal F(x)} << \lambda_x$, i.e, $\mathcal F$ is lower leafwise absolutely continuous with respect to $\mu$. As $\mu$ is $f$-invariant and ergodic, it follows by \cite[Lemma 3.14]{AVWII}\footnote{Without the presence of a $f$-invariant ergodic measure $\mu$, lower leafwise absolute continuity is not, in general, equivalent to leafwise absolute continuity (see \cite{VianaYang} for some examples). 
}  that $\mathcal F$ is leafwise absolutely continuous.
\end{proof}

\section{Proof of Theorem \ref{theorem:1dimens}}
\begin{proof}[Proof of Theorem \ref{theorem:1dimens}]
If the disintegration of $\mu$ is not atomic nor the conditional measures are supported on a Cantor set, then by Theorem \ref{PNV} the conditional measures are equivalent to the measures $\lambda_x$ on $\mathcal F(x)$, induced by $\{d_x\}$. In this case, Theorem \ref{lemma:superaux2} implies that $\mathcal F^c$ is leafwise absolutely continuous. Now, from \cite[Lemma 3.16]{AVWII} it follows that $\mathcal F^{cs}$ and $\mathcal F^{cu}$ are leafwise absolutely continuous. Therefore, by Theorem \ref{theo:KBGeneral} we conclude that $f$ is Bernoulli. 
\end{proof}

\section{Proof of Theorem \ref{theorem:IP}}

\begin{proof}[Proof of Theorem \ref{theorem:IP}]
As in \cite{PNV}, let $\mathcal U$ be an open cover of $M$ by local charts and let $\mathfrak r$ small enough so that for every $x\in M$, there exists $U\in \mathcal U$ with $B_{d_x}(x,\mathfrak r) \subset U$. For more details on how to construct such number see \cite[Proposition 3.9]{PNV}. For a full measure subset $\mathcal Z \subset M$ we may define $\mu_x$ on $\mathcal F(x)$ to be the conditional measure of $\mu$, along $\mathcal F(x)$, normalized so that $\mu_x(B_{d_x}(x,\mathfrak r)) = 1$, for all $x\in \mathcal Z$ (c.f. \cite[Section 4]{PNV}).

For each $x\in \mathcal Z$ and $r\in \mathbb R$ define
\[J(x,r):=\frac{d\mu_x}{d\lambda_x}(x_r^+).\]
It is easy to see that $J(f(x),r)=J(x,r)$, since $f_*\mu_x=\mu_{f(x)}$, $f_*\lambda_x=\lambda_{f(x)}$ and $f(x^+_r) = f(x)^+_r$. Therefore, for a full measure subset $\mathcal M_r$ we have:
\[J(x,r) = J(r).\]
Let $\mathcal N := \bigcap_{r\in \mathbb Q} \mathcal M_r$. By \cite{PNV}, for $x\in \mathcal N$ we have
\[\frac{d\mu^U_x}{d\lambda_x}(y) = \mu^U_x(B_{d_x}(y,\mathfrak r)) \cdot \Delta,\]
where $\Delta$ is a constant, therefore
\[\frac{d\mu_x}{d\lambda_x}(y) = \frac{\mu^U_x(B_{d_x}(y,\mathfrak r))}{\mu^U_x(B_{d_x}(x,\mathfrak r))} \cdot \Delta.\]
Replacing $y$ for $x^+_r$ we have:
\[J(x,r) = \frac{\mu^U_x(B_{d_x}(x^+_r,\mathfrak r))}{\mu^U_x(B_{d_x}(x,\mathfrak r))} \cdot \Delta. \]

The map $x\mapsto \mu^U_x(B_{d_x}(x,\mathfrak r))$ is continuous restricted to plaques of $\mathcal F|U$ and, analogously, the map $(x,r) \mapsto \mu^U_x(B_{d_x}(x^+_r,\mathfrak r)$ is also continuous on the first coordinate restricted to plaques and on the second coordinate restricted to the condition $x^+_r\in U$.
By the continuity of $J(x,r)$ at the first coordinate restricted to plaques, we conclude that for any plaque $L$ intersecting $\mathcal N$ in a full measure set we have $J(x,r)=J(r)$, for every $x\in L$ and any $r\in \mathbb Q$. Therefore, we may assume $\mathcal N$ to be plaque saturated.
Now, for each $x\in \mathcal N$ and $s\in \mathbb R$, given any sequence of rationals $r_n \rightarrow s$ we have:
\[J(x,s) = \lim_{n\rightarrow \infty} J(x,r_n) = \lim_{n\rightarrow \infty} J(r_n).\]
Hence, we conclude that there is a continuous function $J:\mathbb R \to \mathbb R$ such that for every $x\in \mathcal N$,
\[J(x,r) = J(r), \quad r\in \mathbb R.\]


Now, for $x\in \mathcal Z$, consider a neighborhood $\mathcal B$ of $x$ given by the following:
\begin{equation}\label{eq:adaptedchart}
\mathcal B = \bigcup_{y\in W^s_{\delta}(x), z\in W^u_{\delta}(y)} h^u_{y,z}(B_{d_y}(y,\mathfrak r))\end{equation}
where
\[\mathcal B^s = \left(\bigcup_{z\in W^s_{\delta}(x)} h^s_{x,z}(B_{d_x}(x,\mathfrak r)) \right).\]
For each $L \in \mathcal F^c|\mathcal B$ consider $x_L \in L$ such that $L=B_{d_L}(x_L,\mathfrak r)$. On $L$ define:
\[d\widetilde{\mu}_L = \mathfrak J\cdot d\lambda_L,\]
where $\mathfrak J(y):=J(s)$ where $y=(x_L)^+_s$, $s\in (-\mathfrak r, \mathfrak r)$.
In particular  $\widetilde{\mu}_L = \mu_L$ whenever $L\cap \mathcal N \ne \emptyset$, which is the case for almost every plaque in $\mathcal B$. Therefore $\{\widetilde{\mu}_L\}_{L\in \mathcal B}$ is a disintegration of $\mu(\cdot | \mathcal B)$.

It is also clear that this system of measures is continuous, since $\lambda_L$ is continuous and the density function does not depend on the center plaque. Therefore $\{\widetilde{\mu}_L\}_{L\in \mathcal B}$  is a continuous disintegration of $\mu(\cdot | \mathcal B)$. We are left to prove that this system is $u$-invariant ($s$-invariance follows analogously).

Let $x\in \mathcal B$ and $y\in W^u_{\delta}(x)$. The continuous invariant metric system $\{d_x\}_{x\in M}$ is invariant by unstable and stable holonomies, in particular $h^u_{x,y}(x_r^+) = y_r^+$, for every $x\in M$, $y\in \mathcal F^u(x)$ and $r \in \mathbb R$. In particular $h^u_{x,y}((x,x_r^+)) = (y,y_r^+)$.
Therefore,
\begin{align*}
\widetilde{\mu}_{\mathcal F(x)}(x,x_r^+) &= \int_0^r J(s) d\lambda_x(s) = \int_0^r J(s) d\lambda_y(\pi^u_{x,y}(s)) \\
& = \int_0^r J(h^u_{y,x}(t)) d\lambda_y(t) = \int_0^r J(t) d\lambda_y(t) \\
& = \widetilde{\mu}_{\mathcal F(y)}(h^u_{x,y}(x,x_r^+)).\end{align*}
Since this holds for every $r\in \mathbb R$, we conclude that the system is indeed $u$-invariant as we wanted to show.
\end{proof}

\section{Proof of Theorem \ref{theorem:centerregular}}

First assume that the center conditionals of $\mu$ are not atomic nor supported on a Cantor set, that is, they are fully supported and equivalent to $\lambda_x$. Also, by Theorem \ref{theorem:IP} we may take a finite cover $\mathcal U$ of $M$ by local charts restricted to which there is a continuous disintegration of $\mu$ along plaques of $\mathcal F^c$ which is invariant by stable and unstable holonomies. Let $U\in \mathcal U$ be any of this charts and consider $\{\mu^U_x\}_{x\in U}$ this mentioned continuous disintegration of $\mu(\cdot |U)$ along $\mathcal F^c|U$. Once again we denote by $\prec$ the order relation on $\mathcal F^c|U$ induced by its orientation.

Consider now the map $\psi^U: R^U \to M$ given implicitly by
\begin{equation}\label{eq:floow}
\mu^U_x(x,\psi^U(t,x)) = |t|, \quad x\in L \in \mathcal F^c,\end{equation}
where $R_U = \{(x,t) \in U\times \mathbb R: \mu^U_x(\{y \in \mathcal F^c|U(x): x\prec y\}) \geq |t| \}$.
and such that $\psi(t,\cdot)$ preserves the orientation of $\mathcal F^c$ for every $t$ fixed. 

\begin{lemma}
The map $\psi^U$ is a continuous flow.
\end{lemma}
\begin{proof}
Consider $t,s \in \mathbb R$. Let $x,y,z \in L$ be three points in a center leaf $L$ such that
\[b=\psi_t(a),\quad  c=\psi_s(b).\]
Consider $t,s>0$, the other cases are analogous;
By definition, $t=\mu^U_x(x,\psi(t,x))$, $s=\mu^U_x(\psi(t,x),\psi(s,\phi(t,x)))$ and $t+s= \mu^U_x(x,\psi(t+s,x))$. Therefore, since $\mu^U_x$ is equivalent to $\lambda_x$ we have
\[(x,\psi(t+s,x)) = (x,\psi(t,x)) \cup (\psi(t,x),\psi(s,\phi(t,x))) \Rightarrow \psi(s,\phi(t,x))=\psi(t+s,x).\]
Continuity of $\psi$ follows straight forward from the continuity of $\{\mu^U_x\}_{x\in U}$ .
%
%
\end{proof}

Since the system $\{\mu^U_x\}_{x\in U}$ is continuous and invariant by stable/unstable holonomies inside $U$, it follows that $\psi$ is also invariant by the respective holonomies, i.e,
\begin{equation}\label{eq:holo}
h^s_{x,x'}\circ \psi_t(x) =  \psi_t\circ h^s_{x,x'}(x), \quad \text{and} \quad h^u_{x,x'}\circ \psi_t(x) =  \psi_t\circ h^u_{x,x'}(x),\end{equation}
whenever the composition is well defined.

\begin{lemma}\label{lemma:preservesvolume}
The flow $\psi^U_t$ preserves the measure $\mu$.\footnote{Recall that $\psi_t$ is not defined on the whole product space $U \times \mathbb R$, so this property is restricted to $U$. }
\end{lemma}
\begin{proof}
Note that by definition, $\mu^U_x$ is invariant by $\psi_t$. Now, let $B\subset U$ and $t\in \mathbb R$ small enough such that $\psi_t(B)$ is well defined. Then, 
\[\mu^U(\psi_t(B)) = \int \mu^U_x(\psi_t(B) \cap \mathcal F^c|U(x)) d\nu = \int \mu^U_x(B) d\nu= \mu^U(B).\]
%
that is, $\psi_t$ preserves the measure $\mu$.
\end{proof}

Next we prove that the flow $\psi_t$ is $C^{\infty}$ using an argument similar to the argument used in \cite{AVW}, although in our case, since we do not obtain a disintegration of $\mu$ which is globally invariant, we need to use the local accessibility hypothesis in place of accessibility. The proof is obtained from an application of Journ\'e Lemma (Theorem \ref{theo:JN}) after one has concluded that $\psi_t$ is $C^{\infty}$ along $\mathcal F^c$, $\mathcal F^s$ and $\mathcal F^u$ plaques. 

%

\begin{theorem}\label{theo:JN} \cite{Journe}
Let $\mathcal F_1$ and $\mathcal F_2$ be transverse foliations of a manifold $M$ whose leaves are uniformly $C^{\infty}$. Let $\eta:M\to \mathbb R$ be any continuous function such that the restriction of $\eta$ to the leaves of $\mathcal F_1$ is uniformly $C^{\infty}$ and the restriction of $\eta$ to the leaves of $\mathcal F_2$ is uniformly $C^{\infty}$. Then $\eta$ is uniformly $C^{\infty}$.
\end{theorem}

\begin{lemma} \label{lemma:infnityflow}
The flow $\psi^U_t$ is a $C^{\infty}$ flow.
\end{lemma}
\begin{proof}
Let $L$ be a center plaque inside $U$ and let  $t\in \mathbb R$ be such that $\varphi^U_t(x)=x'$ is well defined in $L$. Consider $x_0=x, x_1, x_2, \ldots, x_n=x'$ be an $su$-local-sequence connecting $x$ and $x'$, that is, with $x_i \in U$ for every $0\leq i \leq n$ . 

Let $T_i$ be the center plaque on $U$ containing $x_{i}$, in particular $T_0=T_n=L$. By the invariance of the disintegration inside $U$ we have
\[\mu^{U}_{x_{i+1}} = (h^{\tau_i}_{x_i,x_{i+1}})_*\mu^{U}_{x_{i}}, \quad \tau_i \in \{s,u\}.\]
In particular we have
\[\mu^U_L = (h_{x,x'})_*\mu^U_L, \]
where $h_{x,x'}$ is a composition of stable and unstable holonomies, therefore a $C^1$ diffeomorphism, and is defined from a neighborhood of $x$ onto a neighborhood of $x'$. Since $x'=\psi_t^U(x)$, by the definition of $\psi^U_t$ and \eqref{eq:holo} we have
\[\psi^U_t = h_{x,x'}, \quad \text{restricted to a neighborhood of}\; x.\]
Thus $\psi^U_t$ is $C^1$ along $L$ and, consequently, $\psi^U_t$ is $C^1$ along center plaques.

%
%
%

Now we will prove that $\psi^U$ is uniformly $C^{r}$ along stable and unstable plaques inside $U$. The argument to prove this last part is the same argument from \cite[Lemmas 7.7, 7.8]{AVW}. We briefly repeat the argument here for the sake of completeness.

Consider $\{\mu^s_x : x\in U\}$ the disintegration of the smooth measure $\mu$ along the plaques of $\mathcal F^s$ in $U$. Since this disintegration is continuous (moreover it is also transversely continuous and with $C^{r}$ densities (see \cite[Lemma 7.6]{AVW})), the map $x \mapsto \mu_x^s$ is continuous. Let $t$ be fixed, then since $\psi_t$ preserves $\mu$ (by Lemma \ref{lemma:preservesvolume}) we have
\begin{equation}\label{eq:disintttt}
(\psi_t)_*\mu^s_x = \mu^s_{\psi_t(x)}, \quad \mu-a.e. \; x\in U.\end{equation}
The disintegration on the right side is situated in the foliation box $\psi_t(U)$ and is also continuous. Since $\psi_t$ is a homeomorphism, the disintegration on both sides are continuous and $\mu$ is smooth, \eqref{eq:disintttt} extends to every point of $U$. That is, $(\psi_t)_*\mu^s_x = \mu^s_{\psi_t(x)}$, for every $x\in U$. In particular, since the densities of $\mu^s_x$ are smooth, $\psi_t$ is the solution of an ordinary differential equation along $\mathcal F^s$-leaves with smooth and transversely continuous coefficients. Thus the solutions are as smooth as the coefficients and vary continuously with the leaf. Therefore, $\psi_t$ is uniformly $C^{r}$ along stable plaques inside $U$. Analogously, $\psi_t$ is uniformly $C^{r}$ along unstable plaques inside $U$.
Finally, by Theorem \ref{theo:JN}, for $t$ and $x$ fixed, since any leaf $\mathcal F^{cs}(x)$ is subfoliated by $\mathcal F^c$ and $\mathcal F^s$ and since $(\psi_t)_{|_{\mathcal F^{cs}(x)}}$ is uniformly $C^{r}$ along $\mathcal F^c$ and $\mathcal F^s$-leaves, we conclude that $\psi_t$ is uniformly $C^{r}$ along $\mathcal F^{cs}$-leaves. Applying the same argument to the pair of transverse foliations $\mathcal F^u$ and $\mathcal F^{cs}$ we conclude that $\psi_t$ is indeed $C^{r}$ on $U$ uniformly in $t$. In particular $\psi$ is $C^{r}$ on $U$ as we wanted.
\end{proof}

Since $\psi_t$ is a $C^{r}$ flow on each open chart $U$ and $\mathcal F^c|U$ is composed by orbits of $\psi_t$ then $\mathcal F^c|U$ is a $C^{r}$ foliation. In particular, as the argument is true for each set from a finite cover by local charts we conclude that $\mathcal F^c$ is $C^r$ as we wanted to show. $\square$

\section{Construction of $\varepsilon$-regular coverings} \label{sec:appB}

Along this section we assume that $f$ is a $C^{1+\alpha}$ volume preserving partially hyperbolic diffeomorphism satisfying hypothesis (1) and (2) from Theorem \ref{theo:KBGeneral}. Here we show the few technical adaptations necessary to show the existence of $\varepsilon$-regular covers of $M$ when we assume that $\mathcal F^{cs}$ is leafwise absolutely continuous. The remaining of the argument to obtain the Bernoulli property will be pointed out in Appendix \ref{sec:appA}.

\begin{definition}
A rectangle is a pair $(P,z)$ where $P\subset M$ is a  measurable set equipped with a point $z\in P$ satisfying the following property: for all $x,y \in P$ the local manifolds $W^{u}_{P}(x)$ and $W^{cs}_{P}(y)$ intersect in a unique point inside $P$.

For the sake of simplicity we also refer to $P$ as being the parallelepiped and to $z$ as being a distinguished point chosen inside $P$.
\end{definition}

It is easy to see from the definition that a rectangle $P$ can be identified with the product: 
\[W^u_P(x) \times W^{cs}_P(x),\]
for any $x\in P$.

\begin{lemma}
Let $P$ be a small enough rectangle. Let $m$ be the volume measure preserved by $f$ and let $\{m^{u}_x\}_x$ the conditional measures obtained from the disintegration of of $m$ along $\mathcal F^{u}$ and $\nu^{cs}_x$ the factor measure in $\mathcal F^{cs}(x)$. 
Then, for any $z\in P$, restricted to $P$ we have $\nu^{cs}_z << \lambda^{cs}_z$ and
\[m^{\Pi}_{P}:= m^u_z \times \nu^{cs}_z << m. \]
\end{lemma}
\begin{proof}
Consider $\eta$ be the partition of $P$ by local unstable leaves. Given any subset $B\subset P$ and any $z\in P$, as $\mathcal F^u$ is absolutely continuous we have
\[m(B) = \int_{P/\eta} m^u_x(B) d\nu(y),\]
where $\nu$ is the factor measure on $P/ \eta$ coming from Rohklin's Theorem. Also, by identifying $P/ \eta$ with $\mathcal F^{cs}(z)\cap P$ we have $\nu^{cs}_z << \lambda^{cs}_z$. In particular we may write,
\[m(B) = \int_{\mathcal F^{cs}(z)} m^u_x(B) d\nu^{cs}_z(y).\]
Now, by definition
\begin{equation} \label{eq:csquasi}
m^u_z \times \nu^{cs}_z (B) = \int_{\mathcal F^{cs}(z)} m^u_z(\pi^{cs}_{x,z}(B)) d\nu^{cs}_z(y).
\end{equation}
Since $\pi^{cs}_{x,z}$ is absolutely continuous for $\lambda^{cs}_z$-a.e. $x \in \mathcal F^{cs}(z)$, it is also absolutely continuous for $\nu^{cs}_z$-a.e. point. Thus, if $m(B)=0$ then by $m^u_x(B)=0$ for $\nu^{cs}_z$-a.e. $x$, which means, by the previous observation that  $m^u_z(\pi^{cs}_{x,z}(B))=0$ for $\nu^{cs}_z$-a.e. $x$ and by \eqref{eq:csquasi} it follows that $m^u_z \times \nu^{cs}_z (B) =0$ concluding the proof.
\end{proof}

\begin{definition}\label{defi:regular}
Given any $\varepsilon >0$, an $\varepsilon$-regular covering of $M$ is a finite collection of disjoint rectangles $\mathcal R=\mathcal R_{\varepsilon}$ such that:
\begin{enumerate}
\item $m(\bigcup_{R\in \mathcal R} R) > 1-\varepsilon$
\item For every $R \in \mathcal R$ we have
\[ \left| \frac{m^P_R(R)}{m(R)} -1  \right| < \varepsilon  \]

and, moreover, $R$ contains a subset, $G$, with $m(G) > (1-\varepsilon)m(R)$ which has the property that for all points in $G$,
\[\left|  \frac{dm^P_R}{dm} -1  \right| < \varepsilon.   \]
\end{enumerate}
\end{definition}

The existence of $\varepsilon$-regular covering of connected rectangles is a known fact for the non-uniformly hyperbolic case by a construction of Chernov-Haskell \cite{CH}. However, as observed in \cite{CH, PTV} if $\mathcal F^{cs}$ is absolutely continous the construction can be repeated, \textit{ipsis literis}, changing $\mathcal F^s$ to $\mathcal F^{cs}$ in the construction of \cite{CH}. The next Lemma states that $M$ always admits $\varepsilon$-coverings. The proof resembles the argument used in \cite{CH}, thus we essentially repeat the construction to show that the absolute continuity hypothesis on $\mathcal F^{cs}$ can actually be replaced by almost absolute continuity of the holonomies in the sense of property (2) of Theorem \ref{theo:KBGeneral}. 

\begin{lemma}\label{lema: partition}
Given any $\delta>0$ and any $\varepsilon >0$, there exist an $\varepsilon$-regular covering of connected rectangles $\mathcal R_{\varepsilon}$ of $M$ with $\operatorname{diam}(R) < \delta$, for every $R \in \mathcal R_{\varepsilon}$.
\end{lemma}

We remark that the proof of Lemma \ref{lema: partition} is very similar to the argument used in \cite{CH} but with the center-stable manifold playing the role of the stable manifold in the proof given in \cite{CH}. The fact that $\mathcal F^{cs}$-holonomies are absolutely continuous between \textbf{almost every} pair of transversals requires a technical adaptation which we show below in details.

\begin{proof}[proof of Lemma  \ref{lema: partition}]
Let $\varepsilon>0$ be given. Up to measure $0$, consider a cover of $M$ by a finite number of open charts separated one to the other by a finite number of smooth compact hypersurfaces. In particular, in each of the chosen charts there is a coordinate system which induces an isomorphism between a bounded domain in $\mathbb R^d$ and the respective chart.

Fix a given chart $(U,\varphi)$. 
\begin{itemize}
\item For each $x\in U$, we identify $T_xM$ with $\mathbb R^d$ via $D\varphi(x):T_xM \to \mathbb R^d$.
\item Given $x,y\in U$ and subspaces $L_x \subset T_xM$, $L_y\subset T_yM$, we denote by $\angle (L_x,L_y)$ the angle between the vector subspaces  $D\varphi(x) \cdot L_x$ and $D\varphi(y) \cdot L_y$ of $\mathbb R^d$.
\item Denoting by $\mathcal L$ the Lebesgue measure $\mathcal L$ on $\mathbb R^d$, we set $\lambda = \varphi_{*}\mathcal L$ on $U$. As the chart is smooth $\lambda$ is equivalent to the Riemannian volume $m$ defined by the metric in $M$. Thus $m << \lambda$ and there is a constant $\delta>0$ such that 
\[\lambda(A)<\delta \Rightarrow m(A)<\varepsilon/4.\]
\item The euclidian metric in $\mathbb R^d$ can be pulled back by $\varphi^{-1}$ to a metric in each chart. This metric will be called the \textit{Euclidian metric} on the chart and, since the chart is a smooth function, this metric is strongly equivalent to the Riemannian metric, say with a constant $c$, which can be taken to be smaller than $2$ by making a convenient choice of the charts and of the systems of coordinates.
\end{itemize}

As in \cite{CH}, for $x\in M$ and $\tau \in \{s,u\}$, we denote by $r^{\tau}_x$ the Euclidian distance of $x$ to $\partial \mathcal F^{\tau}(x)$ measured along the manifold $\mathcal F^{\tau}(x)$. For $\alpha>0$, $x\in M$ denote
\[r^{\tau}_x(\alpha) = \min \left\{r^{\tau}_x \; , \; \inf_{y\in \mathcal F^{\tau}(x): \angle (E^{\tau}(y), E^{\tau}(y) ) \geq \alpha} d^{\mathcal F^{\tau}}(x,y) \right\}.\]
As observed in \cite[Pg.16]{CH}, $r^{\tau}_x(\alpha)>0$ for any $x\in M$, $\alpha>0$, and $\tau \in \{s,u\}$.

By Lusin theorem we may take $M_{\varepsilon} \subset M$ a compact subset such that
\begin{itemize}
\item[i)] $m(M_{\varepsilon}) > 1-\varepsilon/4$;
\item[ii)] $x\mapsto E^u_x$ and $x\mapsto E^{cs}_x$ depend continuously on $x\in M_{\varepsilon}$;
\item[iii)] 
\[\overline{\alpha}=\min_{x\in M_{\varepsilon}}\angle (E^u(x), E^{cs}(x)) >0; \]
\item[iv)] \[\overline{r} = \min_{x\in M_{\varepsilon}, \tau \in \{s,u\}} r^{\tau}_x(\beta) >0,\]
where $\beta = \min\{\pi/3, \delta \alpha/8d\lambda(M) \}$.
\end{itemize}

Now we can cover $M_{\varepsilon}$, up to a subset of zero measure, by a finite collection of open sets $\mathcal U$ satisfying:
\begin{itemize}
\item[v)] each set of $\mathcal U$ lies in one chart, which defines a coordinate system in it;
\item[vi)] the angles $\angle (E^u_x,E^u_y), \angle (E^{cs}_x,E^{cs}_y)$ do not exceed $\beta = \min\{\pi/3, \delta \alpha/8d\lambda(M) \}$ for any $x,y \in M_{\varepsilon} \cap U, U\in \mathcal U$.
\end{itemize}

We will now associate to each $U\in \mathcal U$ a point $z=z_U \in U$. For each open set $U\in \mathcal U$ consider an arbitrary point $z'\in U$. By hypothesis, for ${\lambda}^{cs}_{z'} \times {\lambda}^{cs}_{z'} $-almost every pair in $(\mathcal F^{cs}(z') \cap U) \times (\mathcal F^{cs}(z') \cap U)$ the center-stable holonomy between transversals is absolutely continuous. In particular, we may pick $z = z_U\in \mathcal F^{cs}(z') \cap U$ such that\footnote{We remark that in the case where $\mathcal F^{cs}$ is absolutely continuous the choice of $z$ may be arbitrary.}, for ${\lambda}^{cs}_{z'}$-almost every point $w \in \mathcal F^{cs}(z') \cap U$, the center-stable holonomy between $W^u_U(w)$ and $W^u_U(z)$ is absolutely continuous.
We may assume without loss of generality that the local chart $\varphi$ defined in $U$ maps $z$ to the origin. Using this point $z=z_U$ chosen in $U$, we fix a new coordinate system defined by pulling back, through $\varphi$, the coordinate system in $\varphi(U)$ defined by $D\varphi(z)\cdot E^s(z)$ and $D\varphi(z)\cdot E^u(z)$, that is $d^u:=\operatorname{dim}(E^u)$ coordinate axes are mutually orthogonal and their tangents are parallel to $E^u_z$, and the same for $d^s:=\operatorname{dim}(E^s)$. In this new coordinate system we partition $U$ into a lattice of $d-$dimensional boxes (see Figure \ref{fig:partitionofU}) whose sides have length $r>0$, where $r$ is chosen so small that
\begin{itemize}
\item[vii)] $r< \overline{r}/2d$;
\item[viii)] the union of all boxes that lie entirely in $U$ has measure greater than $(1-\varepsilon/4)m(U)$.
\end{itemize}

\begin{figure}[htb!]
\begin{center}
\includegraphics[height=5cm]{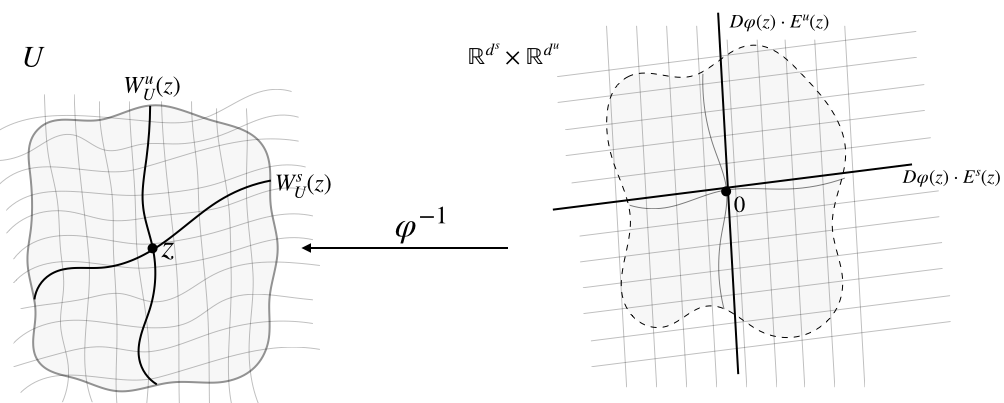}
\caption{Partition of $U$ in small $d$-dimensional boxes whose sides have tangent spaces parallel to $E^u_z$ and $E^s_z$. }
\label{fig:partitionofU}
\end{center}

\end{figure}

The boxes can be made arbitrarily small by decreasing $r$ if necessary. Denote by $\mathcal B$ the collection of all the boxes $B$ such that $B\subset U$, for some $U\in \mathcal U$. The boxes $B\in \mathcal B$ are disjoint and by (viii) we have
\[m\left(\bigcup_{B\in \mathcal B} B\right) > 1-\varepsilon/2.\]
Furthermore, since $\angle (E^{\tau}_y, E^{\tau}_x) < \beta \leq \pi/3$ for all $x,y\in B$ that lie in the same unstable (resp. center-stable) manifold, it follows that the Euclidian distance between $x$ and $y$ measured along the manifold, is less than two times the Euclidian distance between these points. Thus the second condition of the definition of regular covers is satisfied.\\

We call a face of a box $B\in \mathcal B$ , $B\subset U$, a $\tau$-face , $\tau=cs,u$, if it is parallel to $E^{\tau}_z$.  

In each box $B\in \mathcal B$ consider the collection $\mathfrak H$ of all the points $x\in B\cap M_{\varepsilon}$ for which the local manifold $W^{\tau}_B(x)$ does not cross any $\tau$-face of $B$ for $\tau=cs,u$. As these manifolds have length at least $\overline{r}$ so, by our choice of $r$ and since $\beta < \pi/3$ we have
\[\partial W^u_B(x)    \; \text{lies entirely on the cs-face},\]
\[\partial W^{cs}_B(x) \; \text{lies entirely on the u-face},\]
for any such point $x \in \mathfrak H$.
We now complete the set $\mathfrak H$ to a rectangle $\tilde{\Pi}$, which in particular lies inside $B$. We apply this argument to every $B\in \mathcal B$ and call the collection of all those new rectangles, constructed by the last procedure, by $\widetilde{\mathcal  P}$. The construction implies that (see \cite[Pg. 17]{CH}))
\[m \left( \bigcup_{\tilde{\Pi}\in \widetilde{\mathcal  P} } \widetilde{\Pi} \right) >1-3\varepsilon/4.\]

Now we proceed to obtain the measurable properties. Observe that given any rectangle $\widetilde{\Pi}$ we can partition it into a finite number of smaller rectangles by taking partitions of $W_{\widetilde{\Pi}}^u(z)$ and $W_{\widetilde{\Pi}}^{cs}(z)$,
\[W_{\widetilde{\Pi}}^u(z)= \bigcup_{i=1}^{k^u} V^u_z(i), \quad W_{\widetilde{\Pi}}^{cs}(z)= \bigcup_{j=1}^{k^{cs}} V^{cs}_z(j),\]
and taking $\widetilde{\Pi}_z(i,j)$ to be the family of all rectangles generated by $V^u_z(i)$ and $V^{cs}_z(j)$, $1\leq i\leq k^u$, $1\leq j \leq k^{cs}$. We call such a decomposition a \textit{proper partition of $\tilde{\Pi}$}.

Now for $x\in \widetilde{\Pi}$, the $cs$-holonomy map from $W^{u}_R(x)$ to $W^{u}_R(z)$ carries the measure $m^{u}_x$ to a measure on $W^{u}_R(z)$. By the choice of $z=z_U\in U$, this holonomy is absolutely continuous for $m$-almost every $x \in \widetilde{\Pi}$ and then the Jacobian
\[J^{cs}_z(x) = \frac{dm^{u}_z}{d(\pi^c)_* m^{u}_x}\]
is defined at almost every point $x\in \widetilde{\Pi}$ and is an almost everywhere finite and strictly positive measurable function in $x$.

By Lusin's theorem, for any $\varepsilon>0$, in any rectangle $\widetilde{\Pi} \in \widetilde{\mathcal P}$ there is a compact subset $P_{\varepsilon}$ of measure $m(P_{\varepsilon})> (1-\varepsilon^4/10000)m(\widetilde{\Pi})$ on which the $cs$-jacobian, that is, the jacobian of the center-stable holonomy $J^{cs}_z(x)$, is continuous in $x$. Moreover it is bounded on $P_{\varepsilon}$, so that 
\[0<a_{\varepsilon} \leq J^{cs}_z(x) \leq A_{\varepsilon} < \infty,\]
for some constants $a_{\varepsilon}$ and $A_{\varepsilon}$ and all $x\in P_{\varepsilon}$.
By continuity there is a proper partition of each $\widetilde{\Pi}$ such that, for all $\widetilde{\Pi}_z(i,j) \subset \widetilde{\Pi}$ and any $x,y\in \widetilde{\Pi}_z(i,j) \cap P_{\varepsilon}$ we have 
\[|J^{cs}_z(x) - J^{cs}_z(y) | \leq \frac{a_{\varepsilon} \varepsilon}{100}\]
and therefore,
\[\left| \frac{J^{cs}_z(x)}{J^{cs}_z(y)} - 1\right| \leq \frac{\varepsilon}{100}.\]
If $y\in \mathcal F^{cs}(x)$ then
\[\frac{J^{cs}_z(x)}{J^{cs}_z(y)} = J^{cs}_y(x).\]

For any $\widetilde{\Pi} \in \widetilde{\mathcal P}$, consider $\mathcal P_{\varepsilon}$ the collection of all subrectangles $\widetilde{\Pi}_z(i,j)$ for which
\begin{equation}\label{16}
\nu(\widetilde{\Pi}_z(i,j) \cap P_{\varepsilon}) \geq (1-\varepsilon^2/100)\nu(\widetilde{\Pi}_z(i,j)).\end{equation}
Therefore,
\[\nu \left(\bigcup_{\widetilde{\Pi}_z(i,j) \in \mathcal P_{\varepsilon}} \widetilde{\Pi}_z(i,j) \right) \geq (1-\varepsilon^2/100) \nu(\widetilde{\Pi}),\]
so that we does not need to take in consideration the subrectangles $\widetilde{\Pi}_z(i,j)$ that fail to satisfy \eqref{16}.

Finally, for any $\widetilde{\Pi} \in \widetilde{\Pi}$ and any $\widetilde{\Pi}_z(i,j) \in \mathcal P_{\varepsilon}$ there is a point $z(i,j) \in \widetilde{\Pi}_z(i,j)$ such that the Jacobian $J^{cs}_{z(i,j)}(x)$ is sufficiently close to one 
\[|J^{cs}_{z(i,j)}(x)-1| \leq \varepsilon/10,\]
on a subset of points $x\in \widetilde{\Pi}_z(i,j)$ whose measure is at least $(1-\varepsilon/10) \nu(\widetilde{\Pi}_z(i,j))$ in virtue of \eqref{16}.

Integrating the Jacobian $J^{cs}_{z(i,j)}(x)$ inside the rectangles belonging to $\mathcal P_{\varepsilon}$ we obtain
\[\left| \frac{dm^P_{\Pi}}{dm} -1 \right|   = \left| \frac{d(m^{u}_{z(i,j)} \times \nu^{cs}_{z(i,j)})}{dm}-1 \right| < \varepsilon.\]

That is such rectangle satisfies the product property of $\varepsilon$-regular coverings. Also, the measure of all those rectangles is greater than $1-\varepsilon$, so that we obtain an $\varepsilon$-regular covering in $M$ by arbitrarily small rectangles as we wanted. \end{proof}

\section*{Appendix A: Proof of Theorem \ref{theo:KBGeneral}} \label{sec:appA}

Once the construction of the $\varepsilon$-regular covering is done the proof of the Bernoulli property is obtained following the same lines as in \cite{CH} with $\mathcal F^{cs}$ playing the role of $\mathcal F^s$ (similar to the argument used in \cite{PTV2}). In what follows we will describe the scheme of the proof pointing out the steps in which the argument is the same as in \cite{CH, PTV2} and the point in which the Lyapunov stability along the center direction is used. 

\vspace{.5cm}

\noindent \textbf{The basis of the approach:} \quad \\
%

In what follows $\mathbf X=(X,\mu)$ and $\mathbf Y=(Y,\nu)$ are non-atomic Lebesgue spaces, that is, they are both measurably isomorphic to the unit interval $[0,1]$ endowed with the Borel $\sigma$-algebra and the standard Lebesgue measure.

A probability measure $\eta$ on the product space $X\times Y$ is a \textit{joining}  of $X$ and $Y$ if the marginals, or projections, of $\eta$ are $\mu$ and $\nu$, that is, for any measurable sets $A\subset X$, $B\subset Y$ we have
\[\eta(A\times Y) = \mu(A), \quad \text{and} \quad \eta(X\times B) = \nu(B).\] 
We denote by $J(\mathbf X,\mathbf Y)$ the set of all joinings of $\mathbf X=(X,\mu)$ and $\mathbf Y=(Y,\nu)$. 

Let $\alpha = \{A_1,...,A_k \}$ and $\beta = \{B_1,...,B_k\}$ be finite partitions of $X$ and $Y$ respectively. Given $x\in X$, denote by $\alpha(x)$ the atom of $\alpha$ which contains $x$. For $y\in Y$, $\beta(y)$ is defined in a similar way. 

\begin{definition} \label{defi:dbar}
The $\overline{d}$-distance \index{$\overline{d}$-distance! between two partitions} between $\alpha$ and $\beta$ is defined by:
\[\overline{d}(\alpha,\beta) = \inf_{\eta \in J(\mathbf X,\mathbf Y)} \eta \{ (x,y) : \alpha(x) \ne \beta(y)\}. \] 
\end{definition}

Observe that the definition of the $\overline{d}$-distance reflects the idea that we want to measure how small is the set of pairs belonging to atoms of different indexes.

\begin{definition}
Given a sequence of finite partitions $\{\alpha_i\}_1^n$ of $X$, we define the sequence of integer functions $l_i(x)$ by the condition $x\in A^{(i)}_{l_i(x)}$, where $\alpha_i = \{A^{(i)}_{1}, A^{(i)}_{2}, \ldots, A^{(i)}_{n_i}\}$. This sequence of functions $l_i(x)$ is called the $\alpha$-name of the sequence of partitions $\{\alpha_i\}_1^n$.
\end{definition}

Given two sequences of finite partitions $\{\alpha_i\}_{i=1}^n$ and $\{\beta_i\}_{i=1}^n$ of $X$ and $Y$ respectively, a natural way to measure the difference between the $\alpha$-name of a point $x \in X$ and the $\beta$-name of a point $y \in Y$ is to take the function
\begin{equation}
\label{function:h}
h(x,y) = \frac{1}{n} \sum_{i: l_i(x) \ne m_i(y)} 1,\end{equation}
where $\{l_i\}_{i=1}^n$ is the $\alpha$-name of the sequence of partitions $\{\alpha_i\}_1^n$ and $\{m_i\}_{i=1}^n$ is the $\beta$-name of the sequence of partitions $\{\beta_i\}_1^n$.

The \textit{$\overline{d}$-distance} between the sequences of finite partitions $\{\alpha_i\}_{i=1}^n$ and $\{\beta_i\}_{i=1}^n$  is defined by
\[\overline{d}(\{\alpha_i\}_{i=1}^n, \{\beta_i\}_{i=1}^n) = \inf_{\lambda \in J(\mathbf X, \mathbf Y)} \int_{X\times Y} h(x,y) \; d\lambda .\]

A measurable map $\theta:X\rightarrow Y$ is called \textit{$\varepsilon$-measure preserving} if there exists a subset $E \subset X$ such that $\mu(E)\leq \varepsilon$ and for every measurable set $A\subset X\setminus E$,
\begin{equation} \label{eq:epsilonpreserving}
\left| \frac{\mu(A)}{\nu(\theta(A))} - 1 \right| \leq \varepsilon.
\end{equation}

\begin{definition}
Let $f:X \rightarrow X$ be a $\mu$-preserving isomorphism of a measure space $(X,\mu)$. A partition $\alpha$ of $X$ is called a Very Weak Bernoulli partition (VWB) \index{Partition! Very Weak Bernoulli} for $f$ if for any $\varepsilon>0$ there exists $N_0 = N_0(\varepsilon)$ such that for any $N' \geq N \geq N_0$, $n\geq 0$, and $\varepsilon$-almost every element $A\in \bigvee_{k=N}^{N'}f^k\alpha$, we have
\[\overline{d}( \{ f^{-i}\alpha \}_1^n, \{f^{-i}\alpha|A  \}_1^n  )   \leq \varepsilon, \]
where the partition $\alpha|A$ is considered with the normalized measure $\mu/\mu(A)$.
\end{definition}

\begin{theorem} \cite{Ornstein70, Ornstein} \label{theo:ornsteins}
Let $(X,\mathcal B, \mu)$ be a non-atomic Lebesgue space and $f:X \to X$ be a measure preserving automorphism. If there exists a sequence of Very Weak Bernoulli partitions 
\[\varepsilon_1 < \varepsilon_2 < \ldots, \]
with $\operatorname{diam}(\varepsilon_n) \rightarrow 0$. Then $(X, \mu, f)$ is a Bernoulli system.
\end{theorem}

 The Lemma which allows us to do the approach we perform here is the following.

\begin{lemma} \cite[Lemma 4.3]{CH} \label{lemma:functiontheta}
Let $(X,\mu)$ and $(Y,\nu)$ be two nonatomic Lebesgue probability spaces. Let $\{\alpha_i\}$ and $\{\beta_i\}$, $1\leq i \leq n $, be two sequences of partitions of $X$ and $Y$, respectively. Suppose there is a map $\theta: X\to Y$ such that
\begin{itemize}
\item[1)] there is a set $E_1 \subset X$ whose measure is less than $\varepsilon$, outside of which
\[h(x,\theta(x)) < \varepsilon.\]
\item[2)] There is a set $E_2\subset X$ whose measure is less than $\varepsilon$, such that for any measurable set $A\subset X\setminus E_2$
\[\left| \frac{\mu(A)}{\nu(\theta(A))} - 1 \right| < \varepsilon.\]
\end{itemize}
Then 
\[\overline{d}(\{\alpha_i\}, \{\beta_i\}) < c\cdot \varepsilon.\]
\end{lemma}

\vspace{.5cm}

\noindent \textbf{Conclusion of the proof of Theorem \ref{theo:KBGeneral}:} \quad \\

The function $\theta$ required in \ref{lemma:functiontheta} is constructed in the following lemma.

\begin{lemma}\cite[Lemma 4.9]{PTV2} \label{lemma:4.9}
For any $\delta>0$, there exists $0<\delta_1<\delta$ with the following property. Let $\Pi$ be a $\delta_1$-rectangle and $E$ a set intersecting $\Pi$ leafwise. Then we can construct a bijective function $\theta:E\cap \Pi \to \Pi$ such that for every measurable set $F \subset E\cap \Pi$ we have
\[ \frac{m^P_{\Pi}(\theta(F))}{m^P_{\Pi}(\Pi)} = \frac{m^P_{\Pi}(F)}{m^P_{\Pi}(E\cap \Pi)}  \]
and for every $x\in E\cap \Pi$, $\theta(x) \in \mathcal F^{cs}(x)$.
\end{lemma}

The final step is to prove that any partition $\alpha = \{A_1,\ldots, A_k\}$ of $M$ by subsets with piecewise smooth boundaries is very weak Bernoulli. 

Consider such a partition $\alpha$. Given a $\delta$-regular covering of $M$, using Lemma \ref{lemma:4.9}, in \cite[Lemma 4.12, pg.354-357]{PTV2} it is proved that given any $\beta>0$, there exists $\widetilde{N_1}>0$ for which, for any $\widetilde{N}'\geq \widetilde{N} \geq \widetilde{N_1}$ and $\beta$-almost every element $A \in \bigvee_{\widetilde{N}}^{\widetilde{N}'} f^i \alpha$,
there exists a $c\cdot \varepsilon$-measure preserving function $\theta:A\to M$ with
\begin{equation}\label{eq:LS}
\theta(x) \in \mathcal F^{cs}(x) \cap R_i,\end{equation}
where $c$ is a constant independent of $\varepsilon$.

Now, to prove that the Cesaro sum appearing in Lemma \ref{lemma:functiontheta} is small we use Birkhoff theorem and the Lyapunov stable center. Indeed, since $f$ has Lyapunov stable center and $\mathcal F^s$ is contracted, by \eqref{eq:LS} we may take $\delta$ small enough so that if $x,y \in R_i$ and $y\in \mathcal F^c(x) \cap R_i$, then
\begin{equation}\label{eq:LS2}
d(f^n(x),f^n(\theta(x))) < \varepsilon, \quad \forall n\in \mathbb N.\end{equation}
In particular, for $x\in A_{l_i(x)}$ we have by \eqref{eq:LS2} $d(f^i(x), \partial A_{l_i(x)}) < \varepsilon \Rightarrow f^i(x) \in O_{\varepsilon}(A_{l_i(x)})$,
where $O_{\varepsilon}(X)$ denotes the $\varepsilon$-neighborhood of a set $X$. Let  $O_{\varepsilon} = \bigcup_{i=1}^kO_{\varepsilon}(A_i)$.
By Birkhoff Theorem we have
\[\frac{1}{n}\sum_{i=1}^{n}e(l_i(x)-m_i(\theta(x))) \leq \frac{1}{n}\sum_{i=1}^{n}\chi_{O_{\varepsilon}}(f^j(x)) \rightarrow m(O_{\varepsilon} ) , \quad \text{as} \; n\rightarrow \infty.\]
Since $m(O_{\varepsilon} ) \rightarrow 0$ as $\varepsilon \rightarrow 0$, by Lemma \ref{lemma:functiontheta} it follows that $\alpha$ is indeed VWB.

Finally, by taking an increasing sequence of partitions $\alpha_1 < \alpha_2 < \ldots$, each $\alpha_i$ being composed of sets with piecewise smooth boundaries, and such that $\operatorname{diam}(\alpha_i) \rightarrow 0$ we conclude by Theorem \ref{theo:ornsteins} that $f$ is a Bernoulli automorphism as we wanted to show. $\square$

\section*{Appendix B: An extra property of the continuous invariant metric system $\{D_x\}$} \label{app:curiosity}

For future use it may be convenient to have in mind that, with respect to an ergodic invariant measure $\mu$, the supremums taken on the definition of the metric system $\{D_x\}_{x\in M}$, are assumed over $n\in \mathbb N$ for $\mu$-almost every point $x\in M$. More precisely we have:

\begin{proposition}
For $f:M \to M$ and $\{D_x\}$ as in Theorem \ref{lemma:superaux}, given any ergodic $f$-invariant measure $\mu$, there exists an $f$-invariant subset $S\subset X$ of full $\mu$-measure such that for all $x\in S$ we have
\[\sup_{n\in \mathbb N} d^{\mathcal F}(f^n(x),f^n(y))= \sup_{n\in \mathbb Z} d^{\mathcal F}(f^n(x),f^n(y)), \quad \forall \; y\in \mathcal F(x).\]
\end{proposition}
\begin{proof}
Clearly $\sup_{n\in \mathbb N} d^{\mathcal F}(f^n(x),f^n(y)) \leq \sup_{n\in \mathbb Z} d^{\mathcal F}(f^n(x),f^n(y))$ so that it is enough to prove the other hand of the inequality. For each leaf $x\in \mathcal F$, recall the notations
\[D_x (x,y):= \sup_{n\in \mathbb Z} d^{\mathcal F}(f^n(x),f^n(y)), \quad \text{and} \quad B_{x}(x,r) = \{y\in \mathcal F(x) : \; D_x(x,y)<r\}.\]

Consider the sets
\[S^+_r(\varepsilon):=\{x : \quad  \exists  \; n \geq 0, \; |r|-d^{\mathcal F}(f^{n}(x), f^n(x^r))  < \varepsilon \},\]
\[S^-_r(\varepsilon):=\{x :  \quad \exists \; n <0, \; |r|-d^{\mathcal F}(f^{n}(x), f^n(x^r))  < \varepsilon \}.\]

\begin{lemma} $S^+_r(\varepsilon)$ and $S^-_r(\varepsilon)$ are measurable sets.
\end{lemma}
\begin{proof}[proof of the Lemma]
By the claim proved in Theorem \ref{lemma:superaux}, the map $g(x)=d^{\mathcal F}(x,x^r)$ is continuous. Now, observe that since $f(x^r)=f(x)^r$ we have
\[g\circ f^n(x) = d^{\mathcal F}(f^n(x),f^n(x^r)).\]
Thus,
\[S^+_r(\varepsilon) = \bigcup_{n\geq 0} f^{-n}(g^{-1}(|r|-\varepsilon, \infty)),\]
which is a measurable set. Analogous for $S^-_r(\varepsilon)$.
\end{proof}

Obviously 
\[M= S^+_r(\varepsilon) \cup S^-_r(\varepsilon),\]
thus at least one of those sets must have positive measure, say  $\mu(S^+_r(\varepsilon))>0$. By Poincar\'e recurrence, $\mu$-almost every point of $S^+_r(\varepsilon)$ returns to itself, which implies that the set 
%
\[S^{+,\infty}_r(\varepsilon):= \limsup_{n \geq 0} f^{-n}(S^+_r(\varepsilon)) = \{x: \exists \; n_i \rightarrow \infty, f^{n_i}(x) \in S^{+}_r(\varepsilon)\},\]
has the same measure as $S^{+}_r(\varepsilon)$. Since $S^{+,\infty}_r(\varepsilon)$ is $f$ invariant, ergodicity implies that $\mu(S^{+,\infty}_r(\varepsilon))=1$. Now, observe that
\[S^{+}_r(\varepsilon) = \bigcup_{m\geq 0} S^{+}_r(\varepsilon)(m),\]
where 
\[S^{+}_r(\varepsilon)(m)= \{x : |r|-d^{\mathcal F}(f^{m}(x), f^m(x^r))  < \varepsilon \}.\]
Thus, for some $m_0\geq 0$ we must have $\mu(S^{+}_r(\varepsilon)(m_0))>0$. But $f^{2m_0}(S^{+}_r(\varepsilon)(m_0)) \subset S^-_r(\varepsilon)$,
which implies $\mu(S^-_r(\varepsilon))>0$. Again by Poincar\'e recurrence followed by ergodicity we conclude that the set  \[S^{-,\infty}_r(\varepsilon):= \limsup_{n \geq 0} f^{n}(S^-_r(\varepsilon)),\] has full measure as well, that is, $\mu(S^{+,\infty}_r(\varepsilon)) = \mu(S^{-,\infty}_r(\varepsilon))=1$.
Consequently, 
\[\mu(S') =1, \quad \text{where} \quad S':=\bigcap_{r, \varepsilon \in  \mathbb Q_+} [S^{+,\infty}_r(\varepsilon) \cap S^{-,\infty}_r(\varepsilon)] .\]
Let $x\in S'$ and $y\in \mathcal F(x)$. If $D_x(x,y) \in \mathbb Q$ then, by the definition of $S'$,
\[\sup_{n\in \mathbb N} d^{\mathcal F}(f^n(x),f^n(y))= \sup_{n\in \mathbb Z} d^{\mathcal F}(f^n(x),f^n(y)).\]
Now, if $D_x(x,y) \notin \mathbb Q$ consider $(y_n) $ a sequence of points in $\mathcal F(x)$ with $y_n \rightarrow y$ and $D_x(x,y_n) \in \mathbb Q$. Given any $\delta'>0$, by the leafwise equicontinuity of $f$, there exists $n_0 \in \mathbb N$ such that
\[n\geq n_0 \Rightarrow d^{\mathcal F}(f^{-m}(y), f^{-m}(y_n)) < \delta' , \quad \forall m\in \mathbb N.\]
In particular, for $n\geq n_0$,
\begin{align*}
d^{\mathcal F}(f^{-m}(x), f^{-m}(y_n)) \leq & d^{\mathcal F}(f^{-m}(x), f^{-m}(y)) +  d^{\mathcal F}(f^{-m}(y), f^{-m}(y_n))\\
 < & d^{\mathcal F}(f^{-m}(x), f^{-m}(y))+\delta', \quad \forall m \in \mathbb N. \end{align*}
Thus, for each $n\geq n_0$, since $D_x(x, y_n) \in \mathbb Q$ there exists $m_n \in \mathbb N$ for which 
\[d^{\mathcal F}(f^{-m_n}(x), f^{-m_n}(y_n)) \geq D_x(x, y_n)-\delta.\]
Therefore
\[d^{\mathcal F}(f^{-m_n}(x), f^{-m_n}(y)) > D_x(x, y_n)-2\delta, \quad \forall n\geq n_0.\]
Since $\delta'$ is arbitrary we have
\[\sup_{m\in \mathbb Z_-}d^{\mathcal F}(f^{-m}(x), f^{-m}(y)) \geq D_x(x, y_n), \quad \forall n\geq n_0.\]
But clearly, continuity of $f$ implies $\limsup_{n} D_x(x,y_n) \geq D_x(x,y)$, thus $\sup_{m\in \mathbb Z_-}d^{\mathcal F}(f^{-m}(x), f^{-m}(y)) = D_x(x,y)$. That is, for $\mu$-almost every point $x\in S'$, for every $y\in \mathcal F(x)$,  $\sup_{n\in \mathbb N} d^{\mathcal F}(f^n(x),f^n(y))= \sup_{n\in \mathbb Z} d^{\mathcal F}(f^n(x),f^n(y))$,
as we wanted to show. In particular,
\[\mu\left( \bigcap_{\varepsilon \in \mathbb R, r\in \mathbb R} [S^+(r,\varepsilon) \cap S^-(r,\varepsilon) ]\right) =1.\]
Finally, the set
\[S := \bigcap_{n\in \mathbb N} f^{n}\left(  \bigcap_{\varepsilon \in \mathbb R, r\in \mathbb R} [S^+(r,\varepsilon) \cap S^-(r,\varepsilon)] \right),  \]
is $f$-invariant, has full measure and satisfies the requirement of the statement.

\end{proof}

\section{Acknowledgements}
The author acknowledges Ali Tahzibi who introduced him to the main problems being addressed in this paper. This paper was partially written while the author was working as a visitor researcher at Universit\'e Paris-Sud, to whom he greatly thanks for the hospitality. We also acknowledge FAPESP for its financial support through processes \# 2018/25624-0 and \#2022/07762-2.

\bibliographystyle{plain}
\bibliography{Referencias2.bib}

\end{document}